\documentclass{article}

\usepackage{latexsym}
\usepackage{verbatim}
\usepackage{amsmath}
\usepackage{amssymb}
\usepackage{mathrsfs}
\usepackage{amsthm}
\usepackage{tikz}
\usetikzlibrary{matrix}

\usetikzlibrary{decorations.pathreplacing}
\usetikzlibrary{arrows.meta}

\usepackage{amsthm, amsfonts}
\usepackage{euscript}
\usepackage{graphicx}
\usepackage{float}
\usepackage{caption}
\usepackage{tikz-qtree}
\usetikzlibrary{arrows.meta}
\usepackage{subcaption}
\newtheorem{thrm}{Theorem}

\newtheorem{defi}{Definition}[subsection]
\newtheorem{theorem}[defi]{Theorem}
\newtheorem{definition}[defi]{Definition}
\newtheorem{lemma}[defi]{Lemma}
\newtheorem{proposition}[defi]{Proposition}
\newtheorem{conjecture}[defi]{Conjecture}
\newtheorem{corollary}[defi]{Corollary}
\newtheorem{problem}[defi]{Problem}
\newtheorem{remark}[defi]{Remark}
\newtheorem{example}[defi]{Example}

\def\Aut{\mathop{\rm Aut}\nolimits}
\def\Min{\mathop{\rm Min}\nolimits}
\def\Max{\mathop{\rm Max}\nolimits}
\def\Sym{\mathop{\rm Sym}\nolimits}
\begin{document}

\author{Amir Assari,\\ Department of Basic Science,\\ Jundi-Shapur University of Technology,\\ Dezful, Iran
\and Narges Hosseinzadeh,\\ Department of Mathematics, Dezful Branch,\\ Islamic Azad University,\\ Dezful, Iran
\and Dugald Macpherson,\\School of Mathematics,\\University of Leeds,\\Leeds LS2 9JT, UK}

\title{ Set-homogeneous hypergraphs}

\maketitle


\begin{abstract} A $k$-uniform hypergraph $M$ is {\em set-homogeneous} if it is countable (possibly finite) and whenever two finite induced subhypergraphs $U,V$ are isomorphic there is $g\in \Aut(M)$ with $U^g=V$; the hypergraph $M$ is said to be {\em homogeneous} if in addition {\em every} isomorphism between finite induced subhypergraphs extends to an automorphism. 
We give four examples of countably infinite set-homogeneous $k$-uniform hypergraphs which are not homogeneous (two with $k=3$, one with $k=4$, and one with $k=6$). Evidence is also given that these may be the only ones, up to complementation. For example, for $k=3$ there is just one countably infinite $k$-uniform hypergraph whose automorphism group is not 2-transitive, and there is none for $k=4$. We also give an example of a {\em finite} set-homogeneous 3-uniform hypergraph which is not homogeneous.
\end{abstract}

\section{Introduction}
\subsection{Background}
A very rich theory has developed around the notions of {\em  Fra\"iss\'e amalgamation} and {\em homogeneous structure} in the sense of Fra\"iss\'e. At the most basic level the notions concern countable structures over finite relational languages, but the concepts make good sense for languages with function and constant symbols, and there are category-theoretic versions, versions with inverse rather than direct limits, and formulations of amalgamation with respect to a specified class of embeddings, leading to `Hrushovski constructions'. The subject has important connections to permutation group theory, model theory, combinatorial enumeration, Ramsey theory, topological dynamics, and constraint satisfaction. In certain specified binary contexts, there are classification theorems of the homogeneous structures: for example for partial orders \cite{schmerl}, coloured partial orders \cite{truss}, graphs \cite{lachwoodrow}, digraphs \cite{cherlin}, `finite-dimensional permutation structures' (structures in a language with finitely many total order symbols) \cite{braunfeld}, and metrically homogeneous graphs of diameter 3 \cite{amato}. However, there is currently very little (beyond \cite{akhtar}) in the way of classification theorems for homogeneous structures where relation symbols have  arity greater than two, and evidence from binary classifications suggests such results will be very difficult. 

In this paper we focus on $k$-uniform hypergraphs (from now on, called just {\em $k$-hypergraphs}) with $k\geq 3$, and on a slight weakening of homogeneity. Let $L$ be a relational language, and 
let  $t\in \mathbb{N}$ with $t\geq 1$. A relational structure $M$ over $L$  is {\em $t$-set-homogeneous} if $|M|\leq \aleph_0$ (we assume this throughout) and, whenever $U,V\subset M$ are of size $t$ and carry isomorphic substructures, there is $g\in \Aut(M)$ with $U^g=V$; it is {\em ${\leq}t$-set-homogeneous} if it is $s$-set-homogeneous for all $s\leq t$. The structure $M$ is {\em set-homogeneous} if it is $t$-set-homogeneous for all $t\in \mathbb{N}$ with $t>0$. We say a countable structure $M$ is {\em $t$-homogeneous} if {\em any} isomorphism between substructures of $M$ of size  $t$ extends to an automorphism of $M$, and {\em ${\leq}t$-homogeneous} if it is $s$-homogeneous for all $s\leq t$; the structure $M$ is {\em homogeneous} if it is $t$-homogeneous for all $t>0$. A $t$-set-homogeneous structure $M$ is $t$-homogeneous if and only if, for each $A\subset M$ of size $t$, the group induced on $A$ be the setwise stabiliser of $A$ in $\Aut(M)$ is precisely the full automorphism group of the structure induced  on $A$.

We are not sure of the history of set-homogeneity, but there is some discussion of the notion in Section 8 of \cite[Chapter 11]{fraisse}. We remark that Hall's universal locally finite group is often described as the unique countably infinite locally finite group which embeds all finite groups and has the property that any two finite isomorphic subgroups are conjugate. This is a set-homogeneity condition, but in this specific group-theoretic context it  implies the stronger homogeneity condition: any isomorphism between finite subgroups of $H$ is induced by some inner automorphism; see \cite[Lemma 3]{hall}.  

{\em Finite} homogeneous graphs were classified independently by Gardiner \cite{gardiner} and by Golfand and Klin \cite{golfand}. It was shown by Ronse \cite{ronse} that any finite {\em set-homogeneous} graph is homogeneous, and Enomoto \cite{enomoto} gave a very short direct proof of this. Enomoto's argument was shown in \cite[Lemma 3.1]{gray+} to work for finite tournaments, but not for finite digraphs (a directed 5-cycle is set-homogeneous but not homogeneous). Finite set-homogeneous directed graphs (allowing pairs with an arc in each direction) were classified in \cite{gray+}, and \cite{zhou} initiates an analysis of {\em finite} 3-set-homogeneous graphs.  

There is little literature so far on {\em infinite} set-homogeneous structures. Set-homogeneous graphs are considered in \cite{droste+}, where a specific countable infinite graph $R(3)$ related to a circular order is shown to be set-homogeneous but not homogeneous, and it is shown that any countably infinite graph which is ${\leq}8$-set-homogeneous but not ${\leq}3$-homogeneous  is isomorphic to $R(3)$ or its complement. Likewise, in \cite{gray+}, a classification is given of countably infinite set-homogeneous digraphs (not allowing pairs with arcs in both directions) which are not 2-homogeneous. Some  papers of Cameron on multiply-homogeneous permutation groups (\cite{cameron76}, \cite{cameronII}, \cite{cameron4}) have similar flavour -- see Section 2.2 for more on this. The paper \cite{droste-trans} contains several classification results for countable partially ordered sets which are $k$-set-homogeneous (called {\em $k$-transitive} in \cite{droste-trans}) but not $k$-homogeneous and do not contain the pentagon. In a similar spirit, \cite{gray} includes a classification of locally finite graphs with more than one end which are {\em 3-CS-transitive} (whenever $U$, $V$ are {\em connected} induced subgraphs of size at most 3  which are isomorphic, there is an automorphism $g$ with $U^g=V$).

\subsection{The main results}
In this paper, we construct four specific infinite hypergraphs and prove they are set-homogeneous. We then characterise some of these examples up to complementation by low degree transitivity and primitivity  properties of their automorphism group. Before stating our main results, we recall that a permutation group $G$ on a set $X$ is {\em $k$-homogeneous} if it is transitive on the collection of unordered $k$-subsets of $X$, and is {\em $k$-transitive} if it is transitive on the ordered $k$-subsets. (Note a slight inconsistency which seems to be established in the literature; for structures, homogeneity is the stronger of the two conditions  $k$-homogeneity/$k$-set-homogeneity, but for permutation groups, it is the weaker of the conditions $k$-homogeneity/$k$-transitivity.) We say $G$ is {\em $k$-primitive} on $X$ if it is $k$-transitive and for any distinct $x_1,\ldots,x_{k-1}\in X$ the stabiliser $G_{x_1\ldots x_{k-1}}$ acts {\em primitively} on $X\setminus \{x_1,\ldots,x_{k-1}\}$, that is, preserves no proper non-trivial equivalence relation. 

Our main theorems are the following. We stress the assumption throughout that set-homogeneity conditions for us imply by definition that a structure is countable, so our structures are throughout assumed to be countable.
\begin{thrm}\label{3hyper}
\begin{enumerate}
\item[(i)] There is an infinite set-homogeneous but not homogeneous 3-hypergraph $M_3$ whose automorphism group is not 2-transitive, and any infinite ${\leq}4$-set-homogeneous 3-hypergraph whose automorphism group is not 2-transitive is isomorphic to $M_3$.
\item[(ii)] There is an infinite set-homogeneous 3-hypergraph $N_3$ which is 2-homogeneous (so has 2-transitive auomorphism group) but is not 3-homogeneous. 
\end{enumerate}
\end{thrm}

\begin{thrm}\label{4hyper}
\begin{enumerate}
\item[(i)] Any infinite ${\leq}5$-set-homogeneous 4-hypergraph  has 2-transitive automorphism group.
\item[(ii)] There is an infinite set-homogeneous but not homogeneous 4-hypergraph $M_4$ whose automorphism group is not 2-primitive.
\item[(iii)] If $M$ is an infinite ${\leq}5$-set-homogeneous 4-hypergraph such that $\Aut(M)$ is not 2-primitive then either $M$ is isomorphic to $M_4$ or its complement, or $\Aut(M)$ preserves a linear betweenness relation on $M$.

\end{enumerate}
\end{thrm}

\begin{thrm}\label{6hyper}
\begin{enumerate}
\item[(i)] There is an infinite set-homogeneous but not homogeneous 6-hypergraph $M_6$ whose automorphism group is 3-transitive but not 3-primitive.
\item[(ii)] If $M$ is an infinite ${\leq}5$-set-homogeneous 6-hypergraph whose automorphism group is not 3-primitive, then $M$  is isomorphic to $M_6$ or its complement, or $\Aut(M)$ preserves a separation relation on $M$.
\end{enumerate}
\end{thrm}

As noted in Section 5.2, the hypergraphs $M_3, M_4$ and $M_6$ belong in a family -- they live on the same vertex set with  $\Aut(M_3)<\Aut(M_4)<\Aut(M_6)$. We also note in Remark~\ref{two-graph} that $N_3$ is a {\em two-graph} -- a 3-hypergraph such that any 4 vertices carry an even number of edges. The hypergraph $N_3$ is closely related to the set-homogeneous graph $R(3)$ from \cite{droste+} mentioned above. 

The strategy for the construction of $M_3, N_3, M_4, M_6$ is the same in each case, and analogous to the corresponding constructions in \cite{droste+} (for graphs) and \cite{gray+} (for digraphs). Let $M$ be one of these $k$-hypergraphs, viewed as a structure in a language $L$ with a single $k$-ary relation symbol. We first consider a suitable (in each case already known) homogeneous structure $N$ with the same domain but different language $L'$, define $M$ from $N$, and prove that $\Aut(M)=\Aut(N)$ by showing they have the same $\emptyset$-definable relations (in the cases considered, the fact that $\Aut(N)\leq \Aut(M)$ ensures that $M$ is not homogeneous.) We then consider an isomorphism $\sigma:U \to V$ between finite substructures of $M$, and show that the $L'$-structures induced on $U$ and $V$ from $N$ are also isomorphic (though not necessarily via $\sigma$). It follows by homogeneity of $N$ that there is $g\in \Aut(N)$ with $U^g=V$, and since $\Aut(M)=\Aut(N)$ we find $g\in \Aut(M)$, as required. The structures $N$ which yield hypergraphs in this way are chosen from a small family of constructions already known to have interesting properties (see Section 2.1). It appears that other similar constructions do not yield set-homogeneous hypergraphs.  

The results characterising these examples mostly use known results characterising permutation groups with a higher degree of homogeneity than transitivity -- these are discussed in Section 2.2. The characterisation of $M_3$ uses a more direct bare-hands argument.

Lachlan and Tripp \cite{lachlan} classified {\em finite} homogeneous 3-hypergraphs, using the observation  that their automorphism groups are 2-transitive, together with the classification of finite 2-transitive groups (so resting on the classification of finite simple groups). We also briefly consider finite set-homogeneous hypergraphs. We do not carry out a classification (though this looks fully feasible), but obtain the following result. Note that part (ii) shows the Enomoto argument mentioned above does not work for 3-hypergraphs. 

\begin{thrm}
\begin{enumerate}
\item[(i)] For $k\geq 3$, every finite set-homogeneous $k$-hypergraph has 2-transitive automorphism group.
\item[(ii)] There is a  set-homogeneous 3-hypergraph on 7 vertices which is not homogeneous.
\end{enumerate}
\end{thrm}

We also show that aspects of Lachlan's theory of finite homogeneous structures over an {\em arbitrary} finite relational language hold  when homogeneity is weakened to set-homogeneity; see Theorem~\ref{lachlanfinite}.

The paper is organised as follows. In the remainder of this introduction we briefly discuss the model theory of set-homogeneity. 
In Section 2 we give some preliminaries on certain specific treelike homogeneous structures and on some results on permutation groups, mainly of Cameron, that we use. We prove Theorem A in Section 3, and Theorems B and  C in Section 4. We discuss finite set-homogeneous structures (in particular Theorem D) in Section 5.1, and consider some further directions and open questions in Section 5.2.

\subsection{Model theory of set-homogeneity}
We briefly  discuss model-theoretic properties of  set-homogeneity. Most background can be found in \cite{cameron2}, and \cite{tent} can also be taken as a model-theoretic source. 

A countably infinite first-order structure $M$ is {\em $\omega$-categorical} if it is determined up to isomorphism among countable structures by its first-order theory. By the Ryll-Nardzewski Theorem, a countably infinite structure $M$ is $\omega$-categorical if and only if its automorphism group is {\em oligomorphic}, that is, has finitely many orbits on $M^n$ for all $n$. It is immediate that any countably infinite set-homogeneous structure over a finite relational language has oligomorphic automorphism group, so is $\omega$-categorical. 

Recall that the {\em age} {\rm Age}$(M)$ of a countably infinite relational structure $M$ over a finite relational language $L$ is the collection of finite structures which embed in $M$. Fra\"iss\'e's amalgamation theorem states that any two infinite homogeneous structures with the same age are isomorphic, and that a collection of finite $L$-strucures is the age of a homogeneous structure if and only if it is closed under isomorphism, substructure, joint embedding property, and has the `amalgamation property'. It was noted in \cite[Theorem 1.2]{droste+} that we may in Fra\"iss\'e's Theorem  replace `homogeneous' by `set-homogeneous' and `amalgamation property' by `twisted amalgamation property' (TAP), where a class $\mathcal{C}$ of finite structures has (TAP) if and only if for any $A,B_1,B_2\in \mathcal{C}$ and $f_i:A\to B_i$ (for $i=1,2$) there are $D\in \mathcal{C}$, embeddings $g_1:B_1 \to D$, $g_2:B_2\to D$ and $h \in \Aut(A)$ such that $g_1\circ f_1=g_2 \circ f_2\circ h$.

It is well-known that an $\omega$-categorical structure over a finite relational language is homogeneous if and only if its theory has quantifier-elimination. A natural weakening of quantifier-elimination is model-completeness: a theory $T$ is {\em model-complete} if every formula is equivalent modulo $T$ to an existential formula. It is shown in \cite[Chapter 11, Section 8]{fraisse} with an attribution to Pouzet (see also \cite[pp. 90-91]{droste+})  that any infinite set-homogeneous structure $M$ over a finite relational language is {\em uniformly prehomogeneous}; that is, for any finite $A\leq M$ there is finite $B$ with $A\leq B\leq M$ and with $|B|$ bounded as a function of $|A|$, such that for  any partial isomorphism $f$ on $M$ with domain $A$, if $f$ extends to $B$ then $f$ extends to an automorphism of $M$. (Formally, the setting in \cite{fraisse} is for languages with a single relation symbol, but the extension to a finite relational language is routine). It is easily checked that any uniformly prehomogeneous structure is model-complete, and hence that any set-homogeneous $L$-structure is model-complete.

\subsection{Notation} 
If $G$ is a permutation group on a set $X$ (sometimes written as $(G,X)$), we write $x^g$ for the image of $x\in X$ under $g\in G$. For $U\subset X$ and $g\in G$, let $U^g=\{u^g:u\in U\}$. For arbitrary functions we write the function to the left of the argument. 

We view a $k$-hypergraph as a first order structure $(M,E)$, where $E$ is a $k$-ary relation which is assumed only to hold if all arguments are distinct, and to be invariant under permutations of the arguments (i.e. to be {\em irreflexive} and {\em symmetric}).
If $(M,E)$ is a $k$-hypergraph, we shall write $x_1\ldots x_k$ rather than $\{x_1,\ldots,x_k\}$ for an edge. A {\em complete} $k$-hypergraph is a $k$-hypergraph  all of whose $k$-subsets are edges. The {\em complement} of the $k$-hypergraph $(M,E)$, denoted $(M,E)^c$ (or just $M^c$), has the same vertex set $M$, but a $k$-subset of $M$ is an edge of $M^c$ if and only if it is a non-edge of $M$. When we use the word `subhypergraph' we always mean {\em induced subhypergraph}, i.e. the model-theoretic notion of substructure. 

 If $<$ is a total order on $X$, and $A,B\subset X$, we write $A<B$ to denote that
$\forall a\in A\forall b \in B(a<b)$. 

If $(X,\to)$ is a tournament (that is, a digraph such that for any distinct $x,y\in X$ exactly one of $x\to y$ or $y \to x$ holds) and $x\in X$, then $x^+:=\{y\in X: x \to y\}$ and $x^-:=\{y\in X: y \to x\}$. 
We denote by $C_3$ the tournament on $\{a,b,c\}$ such that $a \to b\to c\to a$.

\section{Preliminaries}
We review here some constructions of homogeneous structures and results on permutation groups which we use heavily. 

\subsection{Some homogeneous structures}
As indicated in the Introduction, our strategy for finding set-homogeneous hypergraphs $(M,E)$ is to find certain other very specific homogeneous structures on the same domain and with the same automorphism group. We here give a brief review of the structures used.

First, we recall the linear betweenness relation, circular order, and separation relation which are derivable from a linear order. If $(X,\leq)$ is a linear order, then a {\em linear betweenness relation} $B(x;y,z)$ can be defined on $X$, putting 
$$B(x;y,z) \Leftrightarrow \big((y\leq x\leq z) \vee (z\leq x\leq y)\big).$$
A {\em circular ordering} $K(x,y,z)$ is definable on $X$ with the rule
$$K(x,y,z)\Leftrightarrow \big( (x\leq y\leq z)\vee (y\leq z\leq x)\vee (z\leq x \leq y)\big).$$ Given a circular order $K$ on $X$, a {\em separation relation} on $X$ is defined by 
$$S(x;y;z,w)\Leftrightarrow \big[ \big(K(x,y,z)\wedge K(x,w,y)\big) \vee \big(K(x,z,y)\wedge K(x,y,w)\big)\big].$$
Here, if $x,y,z,w$ are distinct, then $S(x,y;z,w)$ says that $z,w$ lie in distinct segments with respect to $x,y$ of the circular order (and vice versa). Axioms for these can be  found in \cite[Part I]{AN1}, and it can be shown that any structure satisfying these axioms arises from a linear order in this way.

Next, we introduce $C$-relations. Following Section 10 of \cite{AN1}, a   {\em $C$-relation} is a ternary relation $C$ on a set $M$ satisfying (C1)-(C4) of the following axioms (with free variables all universally quantified -- we omit these quantifiers); it is {\em proper} if it also satisfies (C5) and (C6).
\begin{enumerate}
\item[(C1)] $C(x;y,z)\to C(x;z,y)$;
\item[(C2)] $C(x;y,z) \to \neg C(y;x,z)$;
\item[(C3)] $C(x;y,z) \to\big(C(x;w,z)\vee C(w;y,z)\big)$;
\item[(C4)] $x \neq y \to C(x;y,y)$;
\item[(C5)] $\exists x C(x;y,z)$;
\item[(C6)] $x\neq y \to \exists z\big(y \neq z \wedge C(x;y,z)\big)$;
\end{enumerate}
The relation $C$ on $M$ is {\em dense} if also

\noindent
(C7) $C(x;y,z) \to \exists w(C(w;y,z)\wedge C(x;y,w))$. \\
The structure $(M,C)$, where $C$ is a $C$-relation on $M$, is called a {\em $C$-set.}

It is shown in \cite[Theorem 11.2]{AN1} that if $(X,\preceq)$ is a lower semilinearly ordered set (a partial order such that for each $a$ the set $\{x:x\preceq a\}$ is totally ordered, and such that any two elements have a common lower bound), then there is a natural  $C$-relation on the set $S$ of maximal chains (totally ordered subsets) of $X$; here $C(x;y,z)$ holds if and only if $x\cap y \subseteq y \cap z$ (where the chains $x,y,z$ are viewed as subsets of $X$). Furthermore \cite[Theorem 12.4]{AN1} any $C$-set $(M,C)$ arises in this way, with $M$ a `dense' set of maximal chains of some lower semilinear order $(X,\preceq)$ -- the density here means that any node $a\in X$ will lie in some chain  of $M$. The semilinear order $(X,\preceq)$ is canonically constructed from $(M,C)$ -- indeed it is first-order interpretable without parameters in $(M,C)$, as a quotient of an equivalence relation on $M^2$; we shall view $(M,C)$ as coming from such $(X,\preceq)$. Configurations for a $C$-relation (and $D$-relation below) are shown in Figure 1.

\captionsetup[subfigure]{labelformat=empty}
\begin{figure}[H]
	\centering
	\begin{subfigure}[b]{0.25\textwidth}
		\centering

		\begin{tikzpicture}

		\draw (-2,0)--(-3,2)node[above]{$x$};
		\draw (-2,0)--(-1,2)node[above]{$z$};
		\draw (-1.5,1)--(-2,2)node[above]{$y$};
		\end{tikzpicture}

		\captionof{figure}{$C(x;y,z)$}

	\end{subfigure}
	\begin{subfigure}[b]{0.25\textwidth}
		\centering  
		\begin{tikzpicture}

		\draw (4,0)--(6,0);
		\draw (4,0)--(3,1)node[above]{$x$};
		\draw (4,0)--(3,-1)node[below]{$y$};
		\draw (6,0)--(7,1)node[above]{$z$};
		\draw (6,0)--(7,-1)node[below]{$w$};
		\end{tikzpicture}

		\captionof{figure}{$D(x,y;z,w)$}

	\end{subfigure}
	\caption{}
	
\end{figure}

With $(X,\preceq)$ and $(M,C)$ as above, for $a\in X$, let $S_a$ be the set of chains in $M$ which contain $a$. There is a natural equivalence relation $E_a$ on the set $S_a$: we put $E_a xy$ if and only if there is $b\in X$ with $a\prec b$ such that $x,y$ both contain  $b$. We call the $E_a$-classes {\em cones} of $(M,C)$, and say that $(M,C)$ is {\em regular}  if the number of cones at $a$   does not depend on $a$, and is {\em $k$-regular} if this  number is $k$; in this case, following  \cite[Section 10]{AN1}, the {\em branching number} of $(M,C)$ is $k+1$. It is well-known (see for example \cite[pp. 159, 161]{cameron-tree}, or \cite[Theorem 12.6]{AN1})  that for each $k\in \mathbb{N}^{\geq 2} \cup \{\infty\}$ there is up to isomorphism a unique countably infinite dense $k$-regular proper $C$-set, and this structure is homogeneous. 

Next, we briefly introduce $D$-relations, as axiomatised in Part V of \cite{AN1}. An arity 4 relation $D$ on a set $M$ is a proper dense $D$-relation (and $(M,D)$ is a {\em $D$-set}) if axioms (D1)-(D4) below hold, again with universal quantifiers omitted; it is {\em proper} if also (D5) holds, and {\em dense} if (D6) holds.
\begin{enumerate}
\item[(D1)] $D(x,y;z,w) \to \big(D(y,x;z,w)\wedge D(x,y;w,z) \wedge D(z,w;x,y)\big)$;
\item[(D2)] $D(x,y;z,w) \to \neg D(x,z;y,w)$;
\item[(D3)] $D(x,y;z,w) \to \big(D(u,y;z,w)\vee D(x,y;z,u)\big)$;
\item[(D4)] $(x\neq  z \wedge y\neq z) \to D(x,y;z,z)$;
\item[(D5)] (properness) $\big(x,y,z \mbox{~distinct~} \to \exists w(w \neq z \wedge D(x,y;z,w))\big)$; 
\item[(D6)] (density) $D(x,y;z,w) \to \exists u\big(D(u,y;z,w) \wedge D(x,u;z,w) \wedge D(x,y;u,w) \wedge D(x,y;z,u)\big)$.
\end{enumerate}

If $(M,C)$ is a $C$-set, we may define a $D$-relation on $M$, where, for $x,y,z,w$ distinct we put $D(x,y;z,w)$ if
$$\big(C(x;z,w)\wedge C(y;z,w) \big)\vee \big(C(z;x,y)\wedge C(w;x,y)\big).$$
We tend to think of a $D$-relation as holding on the set of `directions' of a general betweenness relation (as defined in Part V of \cite{AN1}); in fact, by \cite[Theorem 26.4]{AN1} any $D$-relation arises in essentially this way. If $(M,E)$ is a graph-theoretic unrooted tree whose vertices have degree at least three, then there is a $D$-relation (not satisfying (D6)) on the set of {\em ends}; for distinct ends $\hat{x},\hat{y},\hat{z},\hat{w}$ we put $D(\hat{x},\hat{y};\hat{z},\hat{w})$ if and only if there are $x\in \hat{x},y\in \hat{y}, z \in \hat{z}$ and $w\in \hat{w}$ such that $x\cup y$ and $z\cup w$ are vertex-disjoint two-way infinite paths. Observe that if $(M,D)$ is a $D$-set  and $a\in M$, then there is an induced $C$-relation $C_a$ on $M\setminus \{a\}$ with `downwards direction $a$' -- define $C_a(x;y,z)$ to hold if and only if $D(a,x;y,z)$ holds. We shall say that the $D$-set $(M,D)$ is {\em $k$-branching} if the corresponding $C$-set $(M\setminus\{a\}, C_a)$ is $k$-branching. Again, for each $k\in \mathbb{N}^{\geq 3} \cup \{\infty\}$ there is a unique countably infinite dense proper $k$-branching $D$-set, and this structure is homogeneous (existence and uniqueness follow from \cite[Theorems 12.6 and 22.1]{AN1}, and homogeneity from the `First Variation' in \cite[p.159]{cameron-tree}). 

There is a natural notion of a $C$-relation $C$ on $M$ being {\em compatible} with a total order on $M$ -- this is introduced and exploited in Section 2.1. There is likewise a notion of a $D$-relation being compatible with a circular order, but this appears not to lead to further set-homogeneous hypergraphs.

We mention one further construction -- the {\em dense local order} of \cite{cameronII}, which is one of the three countably infinite homogeneous tournaments classified by Lachlan in \cite{lachlantourn}. Following \cite{cameronII}, a {\em local order} is a tournament such that all out-neighbourhoods $x^+$ and in-neighbourhoods $x^-$ are totally ordered by $\to$. The dense local order is easiest to describe as the unique countably infinite tournament $T=(Z, \to)$ obtained by distributing a countably infinite set $Z$ of  points densely on the unit circle, no two antipodal, and putting $x \to y$  if the clockwise distance from $x$ to $y$ on the circle is less than the anticlockwise distance. It is the countable homogeneous tournament obtained as the Fra\"iss\'e limit of the collection of all finite local orders; by \cite[Theorem 6.2]{cameronII} it is the unique (up to isomorphism) countable tournament on at least 4 vertices containing $C_3$ and such that each  set $(x^+,\to)$ and $(x^-,\to)$ is a dense linear order without endpoints. 
As noted by Cameron in \cite[pp. 57--58]{cameronII} this tournament is isomorphic to its complement, and if $\gamma$ is such an isomorphism, then the group $H:=\langle \Aut(T), \gamma\rangle$ has $\Aut(T)$ as a subgroup of index 2. In particular, $H$ is 2-transitive, not 2-primitive, and preserves a separation relation on $Z$. 

These examples have been heavily explored in the literature, from various points of view. For example, they
are among the few known examples of countably infinite structures $M$ whose automorphism groups $\Aut(M)$ are oligomorphic,  primitive, and have the property that if $f(k)$ denotes the number of orbits of $\Aut(M)$ on $k$-element subsets of $M$, then $f(k)$ is bounded above exponentially. This viewpoint is developed in \cite{cameron-tree}.  Among homogeneous structures, these structures each have the rare property that their age is well-ordered under embeddability. For recent work in this direction see for example Conjecture 1 of \cite{braunfeld2} that for a homogeneous $\omega$-categorical relational structure these conditions are co-incident, and coincide also with `monadic NIP'. 
The automorphism groups of homogeneous $C$ and $D$-sets are {\em Jordan} groups -- this was a motivating theme of \cite{AN1}.

Many of our arguments with $C$ and $D$-relations and local orders are very pictorial. We include some diagrams to indicate the intended configurations, but encourage the reader to draw others.

\subsection{Homogeneity in permutation groups}
Our results on set-homogeneity use and mimic several earlier results on homogeneity of permutation groups. First, recall the following result of Cameron which lies in the background. 
\begin{theorem} \cite{cameron76}
Let $(G,X)$ be an infinite  permutation group which is $k$-homogeneous for all $k$ but not $k$-transitive for some $k$. Then $G$ preserves on $X$ a linear order, circular order, linear betweenness relation, or separation relation.
\end{theorem}

We use below a number of refinements of this result, listed together below.

\begin{theorem} \label{4hom}
Let $G$ be a  permutation group on an infinite set $X$. Then the following hold. 
\begin{enumerate}
\item[(i)] (J. MacDermott -- see  \cite[p. 63 (3.11)]{cameron2}) Suppose that  $(G,X)$ is 3-homogeneous but not 2-transitive. Then $G$ preserves a linear order on $X$.
\item[(ii)] (Cameron \cite[Theorem 3.3]{cameron4}) Let $(G,X)$ be  3-homogeneous, 2-transitive not 2-primitive. Then   $G$ preserves on $X$ a dense linear betweenness relation  or a structure elementarily equivalent to the universal homogeneous 2-regular $C$-set.
\item[(iii)] (Cameron \cite[Theorem 5.1]{cameron76})   Suppose that $G$ is 4-homogeneous but not 3-transitive. Then $G$ preserves a circular order or linear betweenness relation on $X$. 
\item[(iv)] (Cameron \cite[Remark p. 245]{cameron4} -- see \cite[Proposition 1.3]{mac0} for proof) If $G$ is 5-homogeneous, 3-transitive but not 3-primitive, then either $G$ preserves a separation relation on $X$ or $G$ preserves on $X$ a structure elementarily equivalent to the universal homogeneous 3-branching $D$-set.
\item[(v)] (\cite[Theorem 1.1]{mac0}) Suppose that $k\geq 5$ and $G$ is $(k-1)$-transitive but not $k$-transitive. Then $G$ is not $(k+3)$-homogeneous.
\end{enumerate}

\end{theorem}

\section{Countably infinite set-homogeneous 3-hypergraphs}
In this section we first describe an example of a countably infinite set-homogeneous 3-hypergraph which is not 2-transitive (Subsection 3.1), thereby proving one direction of Theorem~\ref{3hyper}(i). We prove the other direction in  Section 3.2, and part (ii) in Section 3.3.  

\subsection{Construction of an example}
In this subsection we construct from a totally ordered $C$-set a countably infinite set-homogeneous 3-hypergraph $(M,E)$ whose automorphism group $G$ is not 2-transitive.  The ordered $C$-set has appeared previously, in various guises -- see for example the structure $\partial PT_3$ on p.162 of \cite{cameron-tree}, or the structure $\mathcal{M}_2$ in \cite[Section 3]{wom}, or $(\mathbb{L}, C,\prec)$ of \cite[Proposition 3.14]{bodirsky}.

Suppose that on the set $M$ there is defined a total order $\leq$ and a $C$-relation $C$. We say that the relations $C$ and $\leq$ are {\em compatible}, and that $(M,C,\leq)$ is a {\em  $(C,\leq)$-set} if all cones  and all sets $S_a$ (see Section 2.1) are convex with respect to $\leq$. However, we only consider the notion when $(M,C)$ is a 2-regular $C$-set, and under this condition compatibility is equivalent to the condition that whenever $x<y<z$ we have $C(x;y,z)\vee C(z;x,y)$ (cf. \cite[Section 3.5]{bodirsky}); we use the definition of compatibility in this latter form. 
Informally, 
if we draw a lower semilinear order  $(X,\preceq)$ in the plane in the natural way, with maximal chains never `crossing', then the natural left-to-right total order on the set of maximal chains is compatible with the $C$-relation defined above. We shall say that the $(C,\leq)$-set $(M,C,\leq)$ is {\em $2$-regular} if $(M,C)$ is $k$-regular, and is {\em strongly dense} if $(M,C)$ is proper as a $C$-set,  $(M,\leq)$ is a dense linear order without endpoints and in addition

\noindent
(C8) $C(x;y,z) \to \exists w_1,w_2[(C(w_1;y,z)\wedge C(x;y,w_1)\wedge w_1<\Min_{\leq}\{y,z\})\wedge (C(w_2;y,z)\wedge C(x;y,w_2)\wedge \Max_{\leq}\{y,z\}<w_2)]$.\\
One configuration witnessing (C8) is shown in Figure 2. Note that (C8) implies the density condition (C7).

\begin{figure}[H]
	\centering
	\begin{tikzpicture}
	
	\draw (0,0)--(4,3)node[above]{$z$};
	\draw (0,0)--(-1,2)node[above]{$x$};
	\draw (2.7,2)--(2,4)node[above]{$y$};
	\draw[dashed] (1.5,1.2)--(1,3)node[above]{$w_{1}$};
	\draw[dashed] (1.9, 1.4)--(4,2)node[above]{$w_{2}$};
	
	\end{tikzpicture}
	\caption{}
	
\end{figure}

The following result has been long known and attribution is difficult.

\begin{theorem} \label{2-reg}
There is up to isomorphism a unique countably infinite totally ordered strongly dense 2-regular $(C,\leq)$-set $(M,C,\leq)$.
\end{theorem}

\begin{proof} We omit the details. For existence, the existence of the corresponding semilinear order is well-known (it is the countable `2-homogeneous tree' of `positive type' and `ramification order 2' from \cite{droste}). We may take any countable dense set of maximal chains from this semilinear order, with the natural induced $(C,\leq)$-set structure. The structure so obtained is denoted $\partial PT_3$ in \cite[p. 162]{cameron-tree}. Uniqueness can be proved by a routine back-and-forth argument. This is done explicitly in \cite[Theorem 4.6]{wom}, though some translation to the language of that paper is necessary. See also \cite[Proposition 3.14]{bodirsky} (where, as discussed with one of the authors, the condition `strongly dense' was inadvertently omitted).
\end{proof}
For the rest of this section, $(M,C,\leq)$ denotes the structure identified in Theorem~\ref{2-reg}.

We now define the edge relation $E$ to hold of a triple $xyz$ of distinct elements of $M$ if
$$\big(x<\{y,z\} \wedge C(x;y,z)\big)\vee\big(y<\{x,z\}\wedge C(y;x,z)\big)\vee \big(z<\{x,y\} \wedge C(z;x,y)\big).$$
(Here $x<\{y,z\}$ means $x<y \wedge x<z$.) Thus in Figure 2 $xyz$ is an edge but $w_2yz$ is not. Let $M_3$ be the hypergraph $(M,E)$, and put $G=\Aut(M_3)$.

\begin{proposition} \label{exist} The hypergraph $M_3=(M,E)$ is set-homogeneous, but $G$ is not 2-transitive on $M$.
\end{proposition}
\begin{proof}
For convenience, observe that the pictures in Figure 3 correspond to 4-sets carrying 0,1,2,3,4 hypergraph edges respectively.

\captionsetup[subfigure]{labelformat=empty}
\begin{figure}[H]
	\centering
	\begin{subfigure}[b]{0.25\textwidth}
		\centering

		\begin{tikzpicture}

		\draw (0,0)--(-1,1); \draw (0,0)--(1,.5);
		\draw (-.25,.25)--(.7,.7);
		\draw (-.5,.5)--(.5,1);
		\end{tikzpicture}

		\captionof{figure}{$0$ edges}
	\end{subfigure}
	\begin{subfigure}[b]{0.35\textwidth}
		\centering  
		\begin{tikzpicture}

		\draw (4,0)--(3,1); \draw (3.5,.5)--(4,1);
		\draw (4,0)--(5,1); \draw (3.75,.75)--(3.5,1);
		\end{tikzpicture}

		\captionof{figure}{$1$ edge}
	\end{subfigure}
	\begin{subfigure}[b]{0.25\textwidth}
		\centering

		\begin{tikzpicture}

		\draw (7,0)--(6,1); \draw (6.5,.5)--(6.8,1.2);
		\draw (7,0)--(8,1); \draw (7.5,.5)--(7.2,1.2);
		\end{tikzpicture}

		\captionof{figure}{$2$ edges}
	\end{subfigure}
	
	\begin{subfigure}[b]{0.35\textwidth}
		\centering  
		\begin{tikzpicture}

		\draw (10,0)--(9,1); \draw (10.5,.5)--(10,1);
		\draw (10,0)--(11,1); \draw (10.2,.75)--(10.5,1);
		\end{tikzpicture}

		\captionof{figure}{$3$ edges}
	\end{subfigure}
	\begin{subfigure}[b]{0.25\textwidth}
		\centering  
		\begin{tikzpicture}

		\draw (13,0)--(14,1); \draw (13,0)--(12,.5);
		\draw (13.25,.25)--(12.3,.7);
		\draw (13.5,.5)--(12.5,1);
		\end{tikzpicture}

		\captionof{figure}{$4$ edges}
	\end{subfigure}
	\caption{}
	
\end{figure}

First, we observe that the ordering $<$ on $M$ is $\emptyset$-definable in $(M,E)$ and hence $G$-invariant, since
for all distinct $y,z\in M$ we have $y<z$ if and only if there are distinct $u,v\in M\setminus \{y,z\}$ such that the only edges on $\{u,v,y,z\}$ are $uvz$ and $yvz$. Also (though we do not need this), the relation $C$ is definable from $E$: we have $C(x;y,z)$ if and only if
$$\big(y=z\neq x\big)\vee \big((x<\Min\{y,z\})\wedge E(x,y,z)\big) \vee \big((\Max\{y,z\}<x) \wedge (y\neq z) \wedge \neg E(x,y,z)\big).$$
Thus, $\Aut(M,E)=\Aut(M,C,\leq)$. 

We next show that $M_3$ is set-homogeneous. We prove by induction on $|U|$ that if $U$, $V$ are isomorphic finite substructures of $(M,E)$ then they are isomorphic as $(C,\leq)$-sets (possibly with a different isomorphism) and hence by homogeneity of $(M,C,\leq)$ there is $g\in G$ with $U^g=V$. This is immediate by inspection if $|U|\leq 3$.

So suppose that $U,V$ are isomorphic finite substructures of $(M,E)$ with $|U|=|V|\geq 4$, and let $\sigma:U \to V$ be an isomorphism. Define 
$$S=\{x\in U: \forall y, z \in U\setminus \{x\}(y\neq z \to E(x,y,z))\}$$
$$S'=\{x\in V: \forall y, z \in V\setminus \{x\}(y\neq z \to E(x,y,z))\}$$
$$T=\{x\in U: \forall y, z \in U\setminus \{x\}(y\neq z \to  \neg E(x,y,z))\}$$
$$T'=\{x\in V: \forall y, z \in V\setminus \{x\}(y\neq z \to  \neg E(x,y,z))\}.$$
Pictorially, $S$ and $T$ are as in Figure 4. Since $(U,E)\cong (V,E)$, we have $|S|=|S'|$ and $|T|=|T'|$, and also $\sigma(S)=S'$ and $\sigma(T)=T'$.

\begin{figure}[H]
	\centering
	\begin{tikzpicture}
	
	\draw (0,0)--(4,3);
	\draw (0,0)--(-.2,1.7);
	\draw (2,1.5)--(1.8,3);
	\draw (.8,.6)--(.6,2.4);
	\draw (1.5,1.15)--(1.3,2.6);
	\draw [decorate,decoration={brace,amplitude=10pt,raise=-4pt},yshift=0pt]
	(-.5,1.8) -- (2,3.2) node [black,midway,xshift=-0.3cm, yshift=0.4cm]  {\footnotesize $S$};
	\draw[dashed] (1.5,1.4)--(2,1.8);
	\draw (2.7,2)--(2.5,3.1);
	\draw (2.6,2.5)--(2.8,3);
	\draw (2.2,3.1)--(2.5,3);
	
	\draw (3.5,2.6)--(3.3,3.3);
	\draw (3.4,2.95)--(3.7,3.3);

	\draw (10,0)--(6,3);
	\draw (10,0)--(9.8,1.7);
	\draw (9.2,.6)--(9,2.2);
	\draw (8.5,1.15)--(8.2,2.6);
	\draw (7.9,1.6)--(7.5,3);
	\draw (7.1,2.2)--(6.8,3.3);
	\draw (6.95,2.75)--(6.5,3.3);
	\draw (6.4,2.7)--(6.1,3.5);
	\draw (6.25,3.1)--(5.8,3.5);
	
	\draw [decorate,decoration={brace,amplitude=10pt, mirror, raise=-4pt},yshift=0pt]
	(10.1,1.8) -- (7.5,3.2) node [black,midway,xshift=.3cm, yshift=0.5cm]  {\footnotesize $T$};
	\draw[dashed] (8.5,1.4)--(7.75,2.1);
	\end{tikzpicture}
	\caption{}
\end{figure}

Suppose first that $S \neq \emptyset$. Let $S=\{a_1,\ldots,a_r\}$ with $a_1<\ldots <a_r$. Then $a_i<b$ for any $i\in \{1,\ldots,r\}$ and $b \in U\setminus S$ and $C(a_i;a_j,a_k)$ whenever $i<j<k\leq r$, and $C(a_i;a_j,b)$ whenever
$i<j\leq r$ and $b\in U\setminus S$; also $C(a_i;b,c)$ for any $i\in \{1,\ldots,r\}$ and $b,c\in U\setminus S$. The corresponding assertions hold for $S'$ and $V$. Since $\sigma(S)=S'$ we have $\sigma(U\setminus S)=V\setminus S'$. Thus 
$U\setminus S$ and $V\setminus S'$ carry isomorphic hypergraphs of size smaller than $|U|$, so by induction carry isomorphic $(C,\leq)$-sets. It follows that $U$ and $V$ carry isomorphic $(C,\leq)$-sets, so by homogeneity of $(M,C,\leq)$ there is $h \in \Aut(M,C,\leq)$ with $U^h=V$. Such $h$ also preserves $E$ (since $E$ is definable in terms of $C,\leq$), so $h\in G$. 

Thus, we may suppose $S=\emptyset$, and similarly $T=\emptyset$.
There is a partition of $U$ into two  parts $P,Q$, with the following property:

\noindent
(*) $|P|\geq 2$, $|Q|\geq 2$, and for all $p\in P$ and distinct $q,q'\in Q$ we have $E(p,q,q')$, and for all $q\in Q$ and distinct $p,p'\in P$ we have $\neg E(p,p',q)$.\\
 Indeed, in the underlying tree structure induced from 
$(M,C,\leq)$, $P$ will be the left hand cone at the root, and $Q$ the right hand cone. In particular, $P<Q$.

We claim that this is the unique two-part partition of $U$ satisfying (*). Indeed, suppose that $U=P'\cup Q'$ is another such partition, and that $P'$ contains $p,p'$ where $p$ lies in the left hand cone $P$ and $p'$ in the right hand cone $Q$. Now $Q\subseteq P'$; for if $q\in Q\cap Q'$ then $E(p,q,p')$ contradicting that  $q\in Q'$ and $p,p'\in P'$. Thus, if $p''\in Q\setminus \{p'\}$ and $q \in Q'$ then $p''\in P'$ and $q\in P$, so  $E(q,p',p'')$ by (*) applied to $(P,Q)$. This however  contradicts (*) for $P',Q'$. 

It follows that $\sigma$ maps $P$ to the left hand cone $P^*$ of $V$ and $Q$ to the right hand cone $Q^*$ of $V$. In particular the hypergraphs induced on $P$ and $P^*$ are isomorphic, as are those on $Q$ and $Q^*$. Since $|P|<|U|$ and $|Q|<|U|$, by induction
the structures induced on $P$ and $P^*$ are isomorphic as $(C,\leq)$-sets, as are those on $Q$ and $Q^*$. It follows that the $(C,\leq)$-structures induced on $U$ and $V$ are isomorphic, so again by homogeneity of $(M,C,\leq)$ there is $h\in \Aut(M,C,\leq)$ with $U^h=V$, and such $h$ lies in $G$. 
\end{proof}

\subsection{Classification in the not-2-transitive case}
The goal of this subsection is to show that $M_3$ is the {\em unique}  countable set-homogeneous 3-hypergraph whose automorphism group is not 2-transitive. The idea is to recover from $E$ the relations $C$ and $\leq$, show that $(M,C,\leq)$ is a strongly dense 2-regular $(C,\leq)$-set, and apply Theorem~\ref{2-reg}. The first three lemmas below hold for finite as well as countably infinite structures. First, we record the following easy lemma, used throughout the paper. 

\begin{lemma}\label{countedges}
Let $(N,E)$ be a $(k+1)$-set-homogeneous $k$-hypergraph. Then if $U,V\subset N$ with $|U|=|V|=k+1$ and $U$ and $V$ carry the same number of edges of $(N,E)$, there is $g\in \Aut(N,E)$ with $U^g=V$.
\end{lemma}

\begin{proof} Suppose that $U$ and $V$ each have $i>0$ edges. Each edge of $U$ omits one element of $U$,  so the intersection of the edges of $U$ is a subset $S_U$ of size $k+1-i$, and similarly the intersection $S_V$ of the edges of $V$ has size $k+1-i$. It follows that any bijection $\sigma:U \to V$ with $\sigma(S_U)= S_V$ gives an isomorphism $(U,E) \to (V,E)$, and by $(k+1)$-set-homogeneity some such $\sigma$ lifts to $g\in \Aut(N,E)$. 
\end{proof}

The following combinatorial lemma is probably well-known.

\begin{lemma} \label{7vertices}
Let $(T,\to)$ be a ${\leq}2$-homogeneous tournament  such that for each vertex $x$, the 3-cycle $C_3$ embeds in the subtournament $x^+$. Then $C_3$ embeds in each $x^-$.
\end{lemma}

\begin{proof} We assume for a contradiction that $C_3$ does not embed in some (equivalently, any) $x^-$. 
First, $(T,\to)$  has a substructure $\{1,2,3,4\}$  where $2,3,4\in 1^+$ and $2 \to 3 \to 4 \to 2$. By 2-homogeneity the arc $1 \to 3$ lies in a copy of $C_3$ so there is an element $5\in T$ with $3 \to 5 \to 1$. Clearly $5\not\in \{1,2,3,4\}$. We have $4 \to 5$ (otherwise 135 is a $C_3$ in $4^-$) and 
$5 \to 2$ (as otherwise  234 is a $C_3$ in $5^-$).

Similarly, the arc $1 \to 2$ lies in a $C_3$, say $1 \to 2 \to  6 \to 1$. Again, clearly $6\not\in \{1,\ldots,5\}$. We have $3 \to 6$ as otherwise the $C_3$ 126 lies in $3^-$. Similarly, $6 \to 4$ as otherwise 234 is a $C_3$ in $6^-$.

Observe that the arc $2 \to 3$ lies on at least two copies of $C_3$, namely 234 and 235. It follows by 2-homogeneity of $\Aut(T,\to)$ on $T$ that every arc lies on at least two copies of $C_3$.

Next, the arc $1 \to 4$ lies in a $C_3$, say $4\to 7 \to 1 \to 4$. By the last paragraph we may suppose $7\neq 5$, so $7\not\in \{1,2,\ldots,6\}$. We have $2 \to 7$ as otherwise  147 is a $C_3$ in $2^-$, and $7 \to 3$ as otherwise 234 is a $C_3$ in $7^-$.

Next, we determine the orientation of the arcs among 5,6,7. We have $7 \to 6$ as otherwise 426 is a $C_3$ in $7^-$. Also 
$5\to 7$ as otherwise 347 is a $C_3$ in $5^-$. And $6 \to 5$ as otherwise 235 is a $C_3$ in $6^-$.

Thus, $7 \to 6 \to 5 \to 7$ is a $C_3$. This lies in $1^-$, which is the final contradiction.
\end{proof}

\begin{lemma} \label{linear}
Let $(N,E)$  be a ${\leq}4$-set-homogeneous 3-hypergraph such that $K=\Aut(N,E)$ is not 2-transitive. Then $K$ preserves a linear order on $N$.
\end{lemma}
\begin{proof} By set-homogeneity, $K$ is 2-homogeneous on  $N$. Since $K$ is not 2-transitive it preserves a tournament relation $\to$ on $N$; here the set of pairs $(a,b)$ with $a \to b$ is one of the two $K$-orbits on ordered pairs of distinct elements. We assume $\to$ is not a total order, so $(N,\to)$ embeds a copy of $C_3$. Now $|N|\geq 4$, since otherwise $(N,\to)\cong C_3$ and $\Aut(N,E)=S_3$ which is 2-transitive. Thus, $(N,\to)$ embeds paths of length 3 as well as $C_3$, and we may suppose that the copies of $C_3$ are the edges of $(N,E)$. By transitivity of $K$ on $N$,  if for {\em some} $a\in X$ we have that $a^+$ (respectively $a^-$) embeds a copy of $C_3$, then this holds for all $x\in N$. 

If neither $a^+$ nor $a^-$ embeds a copy of $C_3$, then $(N,\to)$ is a local order (see Section 2.1). Easily, if $|a^+|=1$ for all $a$ then $|a^-|=1$ for all $a$ and we have  $(N,\to)\cong C_3$, a contradiction as above. Thus, using rigidity of finite total orders we may assume that each $a^+$ and $a^-$ is infinite. By 3-set-homogeneity each set $a^+$ and $a^-$ is densely ordered, so by \cite[Theorem 6.2]{cameronII} $(N,\to)$ 
 is isomorphic to the dense local order $(S,\to)$ (one of the three homogeneous tournaments classified by Lachlan in \cite{lachlantourn}). Now as noted in Section 2.1, there is a permutation $\gamma$ of $N$ (of order 2) which induces an isomorphism from $(N,\to)$ to its reverse. Now $\gamma$ preserves $E$, so $\gamma\in K$, contradicting that $\to$ is $K$-invariant. 

Thus, we may suppose that each $x^+$ embeds $C_3$. It follows by Lemma~\ref{7vertices} that each $x^-$ also embeds $C_3$. This however means that $(N,\to)$  contains two non-isomorphic 4-vertex tournaments each giving a hypergraph with exactly one edge (a copy of $C_3$ dominated by a vertex, and a copy of $C_3$ dominating a vertex). This is impossible by Lemma~\ref{countedges}. 

\end{proof}

We aim to show that any countably infinite set-homogeneous but not 2-homogeneous 3-hypergraph $(N,E)$ is isomorphic to $M_3$.
Consider such $(N,E)$, and let $K=\Aut(N,E)$. By Proposition~\ref{linear}, $K$ preserves a linear order $\leq$ on $N$, and by 2-homogeneity, this order is dense without endpoints. Our goal is to reconstruct the $C$-relation on $N$.

Let $U=\{x,y,z,w\}$ be a 4-vertex substructure of the structure $(M,E,\leq)$ from Section 3.1, with $x<y<z<w$. It is easily checked (compare Figure 3) that if $U$ has one edge then this is $xyz$, if $U$ has 2 edges then these are $xzw$ and $yzw$, and if $U$ has 3 edges then these are all except $yzw$. Also, all these cases are realised. We first aim to recover this behaviour on $(N,E)$ (of course, up to reversal of the ordering). Note that in $N$, by Lemma~\ref{countedges} any two 4-vertex substructures with the same number of edges are in the same $K$-orbit, so the edges must be distributed in the same way with respect to the order, that is, the structures are isomorphic as $(E,\leq)$-structures. 

We observe first that $(N,E)$ contains both edges and non-edges, since otherwise $K={\rm Sym}(N)$ and so is 2-transitive.
We shall  amalgamate edges and non-edges in all possible ways  over 2-sets. Since we use similar arguments for 4-hypergraphs in the next section, and the arguments may be applicable for $k$-hypergraphs for larger $k$, we set them up in greater generality in the next two lemmas. 

\begin{lemma} \label{5verts}
Let $k\geq 4$ and let $(H,E)$ be a ${\leq}(k+1)$-set-homogeneous countably infinite $k$-hypergraph whose automorphism group  preserves a  total order $<$ on $H$. 
\begin{enumerate}
\item[(i)] Let $i,j\in \{1,\ldots,k+1\}$ such that $|i-j|\geq 2$. Then there are $u_1<\ldots<u_{k+1}$ in $H$ such that
$\{u_1,\ldots,u_{k+1}\}\setminus \{u_i\}$ is an edge and $\{u_1,\ldots,u_{k+1}\}\setminus \{u_j\}$ is a non-edge.
\item[(ii)] Let $i\in \{1,\ldots,k\}$. Then there are $u_1<\ldots < u_{k+1}$ in $H$ such that exactly one of $\{u_1,\ldots,u_{k+1}\}\setminus \{u_i\}$ and $\{u_1,\ldots,u_{k+1}\}\setminus \{u_{i+1}\}$ is an edge.
\end{enumerate}
\end{lemma}

\begin{proof} (i) First choose a non-edge $w_1\ldots w_k$ of $H$ with $w_1<\ldots<w_k$. Suppose first $i<j$. There is $v$ such that $w_{j-1}<v<w_j$ and $\{v,w_1,\ldots,w_k\}\setminus \{w_i\}$ is an edge. 
In the case when $j<i$, choose $v$ with $w_{j-1}<v<w_j$ such that  $\{v,w_1,\ldots,w_k\}\setminus \{w_{i-1}\}$ is an edge.
In both cases put $u_m=w_m$ for $m<j$, $u_j=v$, and $u_m=w_{m-1}$ for $m>j$ to obtain the required set. 

(ii) Again  choose a non-edge  $w_1\ldots w_k$ with $w_1<\ldots<w_k$. Choose $v$ with $w_{i-1}<v<w_{i+1}$ such that
$\{w_1,\ldots,w_k,v\}\setminus \{w_i\}$ is an edge, and list $w_1,\ldots,w_k,v$ in increasing order as $u_1,\ldots,u_{k+1}$. Note that whether or not $v<w_i$ is undetermined. 
\end{proof}

We continue to consider a countably infinite ${\leq}(k+1)$-set homogeneous $k$-hypergraph $(H,E)$ whose automorphism group  is not 2-transitive and so preserves a total order $<$ on $H$. 
 For $i,j\in \{1,\ldots,k+1\}$ with $|i-j|\geq 2$, let $S_{ij}(x_1,\ldots,x_{k+1})$ be a formula (in the language with $<$ and $E$) expressing
$$(x_1<\ldots<x_{k+1}) \wedge  (\{x_1,\ldots,x_{k+1}\}\setminus \{x_i\} \mbox{~ is an edge})$$
$$ \wedge
(\{x_1,\ldots,x_{k+1}\}\setminus \{x_j\} \mbox{~ is a non-edge}).$$
Also for $i=1,\ldots,k$ let
$T_i(x_1,\ldots,x_{k+1})$ express
$$(x_1<\ldots<x_{k+1})$$
$$ \wedge  (\mbox{exactly one of~}\{x_1,\ldots,x_{k+1}\}\setminus \{x_i\} \mbox{~and~}
 \{x_1,\ldots,x_{k+1}\}\setminus \{x_{i+1}\} \mbox{~ is an edge}).$$
The conclusion of Lemma~\ref{5verts} is that the formula $S_{ij}$ is realised in $(H,E)$ for all $i,j\in \{1,\ldots,k+1\}$ with $|i-j|\geq 2$ and $T_i$ is realised for each $i=1,\ldots,k$.

If a $(k+1)$-set has $i$ edges, then these edges intersect in $k+1-i$ elements, and under an assumption of $(k+1)$-set-homogeneity of $(H,E)$ and the invariance of $<$, these $k+1-i$ elements are determined by $i$. 
For each $i=1,\ldots,k$ we say that $H$ satisfies $P^i_J$ where $J$ is a $(k+1-i)$-subset of $\{1,\ldots,k+1\}$, if $H$ has a $(k+1)$-set $\{u_1,\ldots,u_{k+1}\}$ (with $u_1<\ldots <u_{k+1}$) with exactly $i$ hypergraph edges, which intersect in the elements indexed by $J$. (In the case where $i=k$ and so $|J|=1$, we write $P^k_j$ rather than $P^k_{\{j\}}$.) We say $H$ satisfies $P^i_*$ if $H$ has no $(k+1)$-set with $i$ edges. Thus, for each $i$ either $P^i_J$ holds for exactly one $(k+1-i)$-subset $J$ of $\{1,\ldots,k+1\}$, or it holds for no $J$ and $P^i_*$ holds.  Furthermore, by Ramsey's theorem either there is a $(k+1)$-vertex set with $k+1$ edges, or there is a $(k+1)$-vertex set with no edges (possibly both).

\begin{lemma} \label{PkJ}
Assume the conditions of Lemma~\ref{5verts}. Then
\begin{enumerate}
\item[(i)]
$P^k_J$ ensures that all formulas $S_{ij}$ are realised where $i\not\in J$ and $j\in J$, and
\item[(ii)] for $i=1,\ldots,k$, a $(k+1)$-tuple  satisfying $P^k_J$ realises $T_i$ if and only if just one of $i,i+1$ lies in $J$.
\end{enumerate}
\end{lemma}

\begin{proof} If $P^k_J$ holds, then any $u_1< \ldots <u_{k+1}$ carrying  $k$ edges has the edge omitting $u_i$ for each $i\in \{1,\ldots,k+1\}\setminus J$, and has the non-edge omitting $u_j$ for each $j\in J$. Both parts follow. 
\end{proof}

We now revert to our context where $(N,E)$ is a ${\leq}4$-set-homogeneous 3-hypergraph whose automorphism group $K$ preserves a total order  $\leq$ on $N$. For each $i=1,2,3$ either $P^i_J$ holds for some $J\subset \{1,2,3,4\}$ with $|J|=4-i$, or $P^i_*$ holds.

\begin{lemma} 
\label{twocases}
Under the above assumptions, one of the following holds.
\begin{enumerate}
\item[(i)] $P^2_{\{1,2\}}, P^1_{\{2,3,4\}}, P^3_{\{4\}}$
\item[(ii)]  $P^2_{\{3,4\}}, P^1_{\{1,2,3\}}, P^3_{\{1\}}$.
\end{enumerate}
\end{lemma}

\begin{proof} By Lemma~\ref{5verts}, all of the formulas $S_{31}, S_{41}, S_{42}, S_{13}, S_{14}, S_{24}$ must be realised, and it is routine by Lemma~\ref{PkJ} and Table 1  to verify that (i) and (ii)  are the only ways to realise all these $S_{ij}$.
 Conditions $P^2_{\{1,2\}}$ and $P^2_{\{3,4\}}$ each realise three such formulas (namely $S_{31}, S_{41}, S_{42}$, and $S_{13}, S_{14}, S_{24}$ respectively), whilst $P^2_{\{1,3\}}$ and $P^2_{\{2,4\}}$ each realise just one such formula (namely 
$S_{41}$ and $S_{14}$ respectively). Conditions $P^1_{\{2,3,4\}}$ and $P^1_{\{1,2,3\}}$  each realise two formulas $S_{ij}$, and $P^1_{\{1,2,4\}}$ and $P^1_{\{1,3,4\}}$  each realise one. Likewise $P^3_{\{1\}}$ and $P^3_{\{4\}}$ each realise two formulas, and $P^3_{\{2\}}$ and $P^3_{\{3\}}$ each realise one. Thus, to realise all six $S_{ij}$ formulas we must have $P^2_{\{1,2\}}$ or $P^2_{\{3,4\}}$. By examining the possibilities, we easily see that the first case gives (i) and the second case gives (ii).


\tikzset{ 
    table/.style={
        matrix of nodes,
        row sep=-\pgflinewidth,
        column sep=-\pgflinewidth,
        nodes={
            rectangle,
            draw=black,
            align=center
        },
        minimum height=1.5	em,
        text depth=0.5ex,
        text height=1.5ex,
        nodes in empty cells,
        every even column/.style={
            nodes={fill=gray!30}
        },
        column 1/.style={
            nodes={text width=2em,font=\bfseries}
        },
          row 18/.style={
             nodes={
                 fill=black,
                 text=white,
                font=\bfseries
             }
         },
         row 18/.style={
             nodes={
                 fill=black,
                text=white,
               font=\bfseries
            }
        }
    }
}


\begin{table}
\centering

\begin{tikzpicture}

\matrix [table, text width=1.6em]
{
 & $S_{13}$ & $S_{31}$ & $S_{14}$ & $S_{41}$  & $S_{24}$ & $S_{42}$ \\
$P^{1}_{123}$   & & & & 0 & & 0  \\
$P^{1}_{124}$   &  & 0 & & & &   \\
$P^{1}_{134}$   & &  & & & 0 & \\
$P^{1}_{234}$  & 0 & & 0 &  & &  \\
%
$P^{2}_{12}$  & & 0 &  & 0 & & 0 \\
$P^{2}_{13}$  &  &  & & 0 & &  \\
$P^{2}_{14}$  &  & 0 &  & & 0 &   \\
$P^{2}_{23}$  & 0 &  & & &  & 0   \\
$P^{2}_{24}$  &  &  & 0 & & &    \\
$P^{2}_{34}$  & 0 &  & 0 & & 0 &   \\
%
$P^{3}_{1}$  &  & 0 & & 0 & &   \\
$P^{3}_{2}$  & &  & &  &  & 0  \\
$P^{3}_{3}$  & 0 &  & & & &  \\
$P^{3}_{4}$  &  &  & 0 & & 0 &  \\
};

\end{tikzpicture}
\caption{0 in the $(P^k_J,S_{ij})$-entry means that if $H$ satisfies $P^k_J$ then it realises $S_{ij}$.}
\end{table}
\end{proof}

We shall assume that Case (ii) holds; this is justified  since (i) is obtained from (ii) by reversing the order (in fact, if (i) holds, we obtain the complement of $M_3$, which as a hypergraph is isomorphic to $M_3$).

\begin{lemma}\label{i edges}
For each $i\in \mathbb{N}$ with $0\leq i \leq 4$ there is a 4-vertex substructure of $(N,E)$ with exactly $i$ edges.
\end{lemma}
\begin{proof} For $i=1,2,3$ this follows from our Case (ii) assumption. For the case $i=0$,
fix a non-edge $uvw$ of $ $ with $u<v<w$.
There is $a<u$ with $auv$ a non-edge, and by inspecting the possibilities in our Case (ii) assumption, $\{a,u,v,w\}$ has 0 edges. 
For the case $i=4$ consider an edge $pqr$ with $p<q<r$, and  some $d<p$ with $dpq$ an edge -- the set $\{d,p,q,r\}$ must have four edges.
\end{proof}

\bigskip
Next, we define a ternary relation $C$ on $N$ as follows. For $x,y,z\in N$, we put $C(x;y,z)$ if and only if one of the following holds.
\begin{enumerate}
\item[(a)] $y=z \wedge x\neq y$;
\item[(b)] $x,y,z$ are distinct, and $x<\Min\{y,z\}\wedge Exyz$;
\item[(c)] $x,y,z$ are distinct, and $\Max\{y,z\}<x \wedge \neg Exyz$.
\end{enumerate}

\begin{lemma} \label{Crel}
The structure $(N,C)$ is a 2-regular  proper $C$-set.
\end{lemma}

\begin{proof} Axioms (C1) and (C4) of Section 2.1 follow immediately from the above definition. (C2) is immediate if $y=z\wedge x\neq y$, so suppose that $C(x;y,z)$ holds and $x,y,z$ are distinct. Either $x<\Min\{y,z\}$ and $Exyz$, or $\Max\{y,z\}<x \wedge \neg Exyz$.
In the first case, we do not have $y<\Min\{x,z\}$ so $C(y;x,z)$ cannot hold through (b), and as $Exyz$ holds,  $C(y;x,z)$ cannot hold through (c). The argument is similar in the second case. Thus (C2) holds. 

To prove (C3), suppose $C(x;y,z)$ holds, and $w\in N$. In case (a), where $y=z\neq x$, either $w=y=z$ and $C(x;w,z)$ holds, or $w\neq y$ and $C(w;y,z)$ holds, as required. 

Suppose case (b) holds. If $w=x$ then $C(w;y,z)$ holds, if $w=y$ then $C(x;w,z)$ holds, and if $w=z$ then $C(x;w,z)$ via (a). 
Thus, we may suppose $x,y,z,w$ are distinct. Now as $Exyz$ holds, $\{x,y,z,w\}$ has 1,2,3 or 4 edges. If there is just one edge,
then $P^1_{\{1,2,3\}}$ (from Lemma~\ref{twocases}(ii)) yields that $\Max\{x,y,z\}<w$ and $yzw$ is a non-edge, so $C(w;y,z)$ by (c). If there are 2 edges then as $P^2_{\{3,4\}}$ holds  and as $C(x;y,z)$ arises from (b), we must have that $w<\Min\{y,z\}$ and $Ewyz$, so 
$C(w;y,z)$ holds by (b). If there are 3 edges then by $P^3_{\{1\}}$ as $Exyz$ we must have $x<w$. Since all the triples of  $\{x,y,z,w\}$ with element $x$ are edges, we must have $Exwz$, so $C(x;w,z)$ holds by (b). Finally, if there are 4 edges, then if $w<\Min\{y,z\}$ then $C(w;y,z)$ holds, and otherwise $C(x;w,z)$ holds, in each case by (b).

The proof when $C(x;y,z)$ holds through case (c) is similar. Thus, we have established (C3). 

For axiom (C5), if $y=z$ then any $x\neq y$ satisfies $C(x;y,z)$. And if $y\neq z$ then using 2-homogeneity of $K$ and the existence of edges we find $x<\Min\{y,z\}$ with $Exyz$, and then $C(x;y,z)$ holds. 

Finally, for (C6), suppose $x\neq y$. Suppose first $x<y$. Again using 2-homogeneity and existence of edges, there is $z>y$ with $Exyz$, and then $y\neq z \wedge C(x;y,z)$ holds. Similarly, if $y<x$, there is $z<y$ with $\neg Exyz$, and again 
$y\neq z \wedge C(x;y,z)$, the latter by (c).

For 2-regularity, suppose for a contradiction that there are distinct $x,y,z$ such that 
$\neg C(x;y,z)\wedge \neg C(y;x,z)\wedge \neg C(z;x,y)$. We may suppose $x<y<z$. Then $xyz$ is a non-edge by (b) as $\neg C(x;y,z)$, but $xyz$ is an edge by  (c)  as $\neg C(z;x,y)$, a contradiction.
\end{proof}

\begin{lemma} \label{compat}
The order $\leq$ on $N$ is compatible with $C$, and $(N,C,\leq)$  is a strongly dense $(C,\leq)$-set.
\end{lemma}

\begin{proof}
In this 2-regular context, for compatibility we must show that if $x<y<z$ then $C(x;y,z)\vee C(z;x,y)$. This is immediate -- if $xyz$ is an edge then $C(x;y,z)$ holds, and if $xyz$ is a non-edge then $C(z;x,y)$ holds. 

It remains to prove strong density, {\em i.e.} that (C8) holds. So suppose that $C(x;y,z)$ holds. Again, we may suppose $x<\Min\{y,z\}$, the case when $\Max\{y,z\}<x$ being handled similarly.

We suppose first $y=z$.  Then there is $w_1$ with $x<w_1<y$ and $xw_1y$ an edge, and for such $w_1$ we have $C(w_1;y,y)\wedge C(x;y,w_1)$, and likewise there is $w_2>y$ with $xyw_2$ an edge, again yielding $C(w_2;y,y)\wedge C(x;y,w_2)$.

Suppose now $y\neq z$. We may suppose $y<z$.  By Lemma~\ref{i edges} there are $p<q<r<s$ in $N$ such that $\{p,q,r,s\}$ has 4 edges. Since $xyz$ is an edge, by 3-set-homogeneity of $(N,E)$ we may choose $g\in K$ with $\{p,r,s\}^g=\{x,y,z\}$ and put $w_1=q^g$. By $K$-invariance of $<$ we have $x<w_1<y$ and $xyw_1$ and $w_1yz$ are edges, so $C(w_1;y,z)\wedge C(x;y,w_1)$ holds.
  Likewise there is $w_2>z$ such that $\{x,y,z,w_2\}$ has three edges, yielding again $C(w_2;y,z)\wedge C(x;y,w_2)$. 
\end{proof}

{\em Proof of Theorem~\ref{3hyper}(i).}  The existence assertion  follows from Proposition~\ref{exist}. 
For the uniqueness characterisation, suppose that $(N,E)$ is a ${\leq}4$-set-homogeneous countably infinite 3-hypergraph whose automorphism group $K$ is not 2-transitive. By Lemma~\ref{linear} and Lemmas~\ref{Crel} and \ref{compat}, there are a $K$-invariant total order $\leq$ on $N$ and compatible $K$-invariant 2-regular $C$-relation $C$ on $N$ so that $(N,C,\leq)$ is strongly dense. It follows from Theorem~\ref{2-reg} that $(N,C,\leq) \cong (M,C,\leq)$. By our assumption that Case (ii) of Lemma~\ref{twocases} holds,  it follows that $E$ is defined from $C$ and $\leq$ in $N$ in the same way as in $M$, and hence that $(N,E)\cong M_3$. If instead we had assumed that any  4-vertex set in $N$ with two edges has the edges intersecting in the first two elements, then $(N,E)$ would be isomorphic to the hypergraph complement $M_3^c$ of $M_3$. However, since $M_3\cong M_3^c$  (as $(M,C,\leq)\cong (M,C,\geq)$), it again follows that $(N,E)\cong M_3$. \hfill  $\Box$

\subsection{A further set-homogeneous 3-hypergraph.}
We consider a further set-homogeneous 3-hypergraph associated with the countably homogeneous local order $T=(Z,\to)$ described in Section 2.1. Let $H$ be the group described there of automorphisms and anti-automorphisms of $T$, which has $\Aut(T)$ as a subgroup of index 2. Our arguments below heavily appeal to the description of $(Z,\to)$ as consisting of points on the unit circle (see Section 2.1). We frequently use that in this representation of $T$, $\Aut(T)$ preserves the natural circular ordering on $Z$, and $H$ preserves the induced separation relation $S$; this is easily verified. There is a natural notion of a subset $U$ of $Z$ being {\em convex} with respect to the circular order; this means that for any distinct $x,y\in U$ and distinct $z,w\in Z\setminus U$, $\neg S(x,y;z,w)$ holds. 

We define a ternary relation $R$ on $Z$, putting $R(x;y,z)$ if and only if 
$$\big( y \to x \wedge y \to z \wedge x \to z\big) \vee \big( z \to x \wedge z \to y \wedge x \to y\big).$$
Clearly $H \leq \Aut(Z,R)$. When discussing $(Z,R)$ (and $(Z,E)$ below) we often refer to the underlying tournament from which they are defined -- even though these structures do not determine the tournament relation $\to$. 

The lemma below is useful to us and may have independent interest: the group $H$ is significant, since among primitive oligomorphic groups which are not $k$-homogeneous for all $k$,  $H$ has the slowest known growth for the function $f(k)$ mentioned in Section 2.1. 

\begin{lemma} \label{localorder}
\begin{enumerate}
\item[(i)] $H=\Aut(Z,R)$.
\item[(ii)] The structure $(Z,R)$ is homogeneous. 
\end{enumerate}
\end{lemma}

\begin{proof}
(i) Since $H\leq \Aut(Z,R)$, it suffices to show $\Aut(Z,R)\leq H$.  Suppose that $g\in \Aut(Z,R)$. Using that $H$ is transitive on $Z$ and on $\{(x,y,z):R(x;y,z) \mbox{~holds}\}$, there is $h\in H$ and $a,b,c\in Z$ with $R(b;a,c)$ such that $gh$ fixes $a,b,c$. We may suppose (adjusting $h$ if necessary) that $b\to a$. Now for any $u\in Z\setminus\{a,b,c\}$ we have
$\big(b \to u\big)\Leftrightarrow \big( R(a;b,u) \vee R(u;b,a)\big)$, and for distinct such $u,u'$ we have $u \to u' \Leftrightarrow R(u;b,u')$. Also if $b \to u$ and $u'\to b$, then $u\to u'\Leftrightarrow \neg R(b;u,u')$. Orientations of pairs $u,u'$ with $u \to b$ and $u' \to b$ and pairs involving $a,c$ are likewise easily recoverable. Thus, as $gh\in \Aut(Z,R)$, also $gh \in \Aut(Z,\to)<H$, so $g\in H$.

(ii) Let $\alpha:(U,R)\to (V,R)$ be an isomorphism  between finite substructures of $(Z,R)$, and let $a \in Z\setminus U$. We must extend $\alpha$ to $U \cup \{a\}$ (for then a back-and-forth argument suffices). Using the action of $H$, $(Z,R)$ is $\leq 3$-homogeneous, so we may assume $|U|\geq 4$.

We shall say finite $X\subset Z$ is {\em linear} if there is an enumeration  $X=\{x_1,\ldots,x_t\}$ such that $R(x_j;x_i,x_k)$ holds whenever $i<j<k$ or $k<j<i$. We call $(x_1,\ldots,x_t)$ a {\em linear enumeration} of $X$, and note that a linear set has two linear enumerations. (In the presentation of $T$ as consisting of points on the unit circle, a linear set is one contained in a segment making an angle less than $\pi$ at the centre.)

First suppose that $U$ is linear, with linear enumeration $(u_1,\ldots,u_t)$. Let $v_i=\alpha(u_i)$ for each $i$. Then since the linearity is determined by $R$ which is $\alpha$-invariant, $(v_1,\ldots,v_t)$ is a linear enumeration of $V$. By considering the representation in the unit circle, there are two possibilities:

(a) There is $p$ such that $(u_1,\ldots,u_p,a,u_{p+1},\ldots,u_t)$ is a linear enumeration of $U\cup \{a\}$ (we allow here $(a,u_1,\ldots,u_t)$ and $(u_1,\ldots,u_t,a)$).

(b) There is $p$ with $1\leq p<t$ such that $R(u_i;a,u_j)$ whenever $1\leq i<j\leq p$, $R(u_j;u_i,a)$ whenever $ p+1\leq i<j\leq t$, and $R$ does not hold among $a,u_i,u_j$ for $1\leq i\leq p<j\leq t$.

In case (a) it is clear that there is $b\in Z\setminus V$ such that $(v_1,\ldots,v_p,b,v_{p+1},\ldots,v_t)$ is linear, and we extend $\alpha$ by putting $\alpha(a)=b$. If (b) holds, choose $b\in Z\setminus V$ such that $R$ does not hold on $\{b,v_p,v_{p+1}\}$. Again, we may put $\alpha(a)=b$. 

Thus, we may suppose that $U$ is not linear.
 Let $C$ be a  subset of $U$ of maximal size such that $C=\{c_1,\ldots,c_r\}$ and $(a,c_1,\ldots,c_r)$ is a linear enumeration of $C\cup \{a\}$. Put $D=U\setminus C$. Then $D\cup \{a\}$ has a linear enumeration 
$(a,d_1,\ldots,d_s)$. By maximality of $C$ we have $s\leq r$, and as $|U|\geq 4$ we have $r\geq 2$. 
Note that if $1\leq i<j\leq r$ and $1\leq k\leq s$ then $R(a;c_j,d_k)\to R(a;c_i,d_k)$, and if $1\leq i<j\leq s$ and $1 \leq k\leq r$ then $R(a;c_k,d_j)\to R(a;c_k,d_i)$. Also, no ordering of $\{a,c_r,d_s\}$ satisfies $R$, since otherwise $U\cup \{a\}$ is linear with linear enumeration $(d_s,\ldots,d_1,a,c_1,\ldots,c_r)$, contradicting non-linearity of $U$.  Non-linearity of $U$ likewise ensures $R(a;c_1,d_1)$. See Figure 5.

\begin{figure}[H]
\centering
\begin{tikzpicture}
\draw (0,0) circle [radius=1];
\draw (-.4, .86)--(-0.6, .89);
\draw (-0.5, 1)node[left]{$c_{2}$};
\draw (-.7, .62)--(-0.9, .66);
\draw (-.75, .75)node[left]{$c_{1}$};
\draw (-1, .355)--(-0.85, .345);
\draw (-1, .345)node[left]{$a$};
\draw (-1.1, -.22)--(-.9, -.19);
\draw (-1, -.19)node[left]{$d_{1}$};
\draw (-.95, -.55)--(-.8, -.48);
\draw (-.8, -.62)node[left]{$d_{2}$};
\draw (.95, -.55)--(.8, -.48);
\draw (1.5, -.55)node[left]{$d_{s}$};
\draw (.8, .50)--(.95, .55);
\draw (.9, .59)node[right]{$c_{r}$};

\draw (4,0) circle [radius=1];
\draw (3.6, .86)--(3.4, .89);
\draw (3.5, 1)node[left]{$c'_{2}$};
\draw (3.3, .62)--(3.1, .66);
\draw (3.25, .75)node[left]{$c'_{1}$};
\draw (2.9, -.22)--(3.1, -.19);
\draw (3, -.19)node[left]{$d'_{1}$};
\draw (3.05, -.55)--(3.2, -.48);
\draw (3.2, -.62)node[left]{$d'_{2}$};
\draw (4.95, -.55)--(4.8, -.48);
\draw (5.5, -.55)node[left]{$d'_{s}$};
\draw (4.8, .50)--(4.95, .55);
\draw (4.9, .59)node[right]{$c'_{r}$};

\end{tikzpicture}
\caption{}
\end{figure}


 Let $c_i'=\alpha(c_i)$ and $d_j'=\alpha(d_j)$ for each $1\leq i \leq r$ and $1\leq j \leq s$. 
The natural linear order induced from the listing $(c_1,\ldots,c_r,d_s,\ldots,d_1)$ induces a separation relation (i.e. a circular order up to reversal -- see Section 2.1) which agrees the separation relation induced from the one on $Z$ determined by $\to$. This separation relation on $U$ is determined by $R$ which is $\alpha$-invariant,  and hence $(c_1',\ldots,c_r',d_s'\ldots,d_1')$ has the same separation relation.  For convenience we suppose that both $(c_1,\ldots,c_r,d_s,\ldots,d_1)$ and $(c_1',\ldots,c_r',d_s',\ldots,d_1')$ are cyclically ordered clockwise around the unit circle as in Figure 5, and therefore may refer to the underlying tournament $(Z,\to)$ as a convenient way of indicating angles at the centre.

Now for any $1\leq i \leq r$ and $1\leq j\leq s$, suppose $(i,j) \neq (r,s)$, say $i\neq r$. Then
$R(a;c_i,d_j)\leftrightarrow \neg R(c_r;c_i,d_j)$. Since $R(c_r;c_i,d_j)\leftrightarrow R(c_r';c_i',d_j')$, it suffices to show there is $b\in Z\setminus V$ such that
$R(b;c_1',d_1')$ and both $(b,c_1',\ldots, c_r')$ and $(b,d_1',\ldots,d_s')$ are linear enumerations. Now since $\neg R(c_r;c_1,d_1)$ (as $U$ is not linear), we have $\neg R(c_r';c_1',d_1')$, so the clockwise angle from $d_1'$ to $c_1'$ at the centre is less than $\pi$, that is, $d_1'\to c_1'$. Also, by considering angles at the centre, we have ${c'_r}^{+}\cup {d'_s}^{-}\neq T$, and as $c'_1\to c'_r$ we have $c'_1\not\in {c'_r}^{+}$. Likewise $d'_1\not\in {d'_s}^{-}$ (we allow $1=s$). It follows that we may choose $b\in Z\setminus V$ with $R(b;d_1',c_1')$ and $b\in {d'_s}^+ \cap {c'_r}^{-}$, and for such $b$ the extension of $\alpha$ with $\alpha(a)=b$ has the required properties.

\end{proof}

Let $L_R$ be the language with just the relation symbol $R$. Define a 3-hypergraph $N_3=(Z,E)$ whose edges are the 3-sets of $(Z,R)$ which satisfy $R$ under some ordering, that is, lie in a segment making an angle less than $\pi$ at the centre. 

\begin{proposition}\label{N3}
The hypergraph  $N_3$ is set-homogeneous, but not 3-homogeneous. In particular its automorphism group is not 2-primitive, and does not induce the full symmetric group on triples satisfying $E$.
\end{proposition}

\begin{proof} We first show that $G=\Aut(N_3)$ preserves $R$, and hence equals $H$. Indeed, if $xyz$ is an edge, then we have
$$R(y;x,z)\leftrightarrow[\exists u\exists v( uxy, vyz \mbox{~are edges and~} uxz, vxz \mbox{~are non-edges})].$$
In particular, as $R(x;y,z) \to \neg R(y;x,z)$, $(Z,E)$ is not 3-homogeneous. Also, $G$ is not 2-primitive, since $G_x$ preserves an equivalence relation on $Z\setminus \{x\}$ with classes $\{y: x \to y\}$ and $\{y: y \to x\}$; indeed, distinct $y,z$ are in the same equivalence class if and only if $R(y;x,z)\vee R(z;x,y)$ holds.

Next, we show
that $N_3$ is set-homogeneous. For any finite $U\subset Z$, define $\sim_U$ on $U$, putting 
$$a\sim_U b \Leftrightarrow \big((a=b) \vee (\forall x\in U\setminus \{a,b\}) abx \mbox{~is an edge}\big).$$
 It is easily seen that $\sim_U$ is an equivalence relation on $U$, and its classes are convex in the (clockwise) circular order induced from $N_3$  and are complete subhypergraphs. We say that $U$ is {\em balanced} if all $\sim_U$-classes have size 1.

{\em Claim 1.} Suppose that finite $U\subset Z$ is balanced. Then $|U|$ is odd and $\Aut(U,E)=\Aut(U,R)$ and equals the dihedral group $D_{|U|}$.

{\em Proof of Claim.} For any distinct $x,y\in U$ with $x\sim_U y$ and $x\to y$, there is $z\in U$ such that $x$ and $y$ lie in opposite segments with respect to $z$, that is, we have $y\to z \to x$. Furthermore, assuming further  there is no $w\in U$ with $x\to w\to y$, such $z$ is unique; for if $z,z'\in y^+\cap x^-$ then $z \sim_U z'$. Also, if $z\in U$ and $z^+ \cap U=\{x_1,\ldots,x_r\}$ with $z=x_0\to x_1 \to \ldots\to x_r$, then for each $i=0,\ldots,r-1$ there is unique $y_i\in z^-\cap U$ with $x_{i+1} \to y_i\to x_i$, and $U=\{x_0,x_1,\ldots,x_r,y_1,\ldots,y_r\}$, so $|U|=2r+1$.

It is now easily seen that with $D_{2r+1}$ acting in the natural
way on $U$ preserving the induced separation relation $S$, we have $D_{2r+1} \leq \Aut(U,R)\leq \Aut(U,E)$. To see that $\Aut(U,E)\leq D_{2r+1}$, suppose $g\in \Aut(U,E)$ fixes $x_0, x_i$ with $i>0$ (in the notation above). We claim that $g$ fixes $x_1$, so suppose $i>1$. Indeed, $g$ fixes setwise 
$\{x_j:0< j<i\}=\{w\in U: (\forall y\in U)(Eyx_0x_i\to Eyx_0w)\}$. Thus $g$ fixes $y_1$, which is the unique element of $U\setminus \{x_0,\ldots,x_i\}$ such that there is exactly one $w\in \{x_j: 0<j\leq i\}$ (namely $x_1$) with $\neg Ey_1x_0w$. Hence $g$ fixes $x_1$. 

Now $y_1\in U$ is unique in $U$ such that $x_0x_1y_1$ is a non-edge, so $g$ fixes $y_1$. Now $x_2$ is unique in $U\setminus\{x_0,x_1\}$ such that $|\{y\in U: x_1x_2y \mbox{~is a non-edge}\}|=1$, so $g$ fixes $x_2$. Continuing this way, we find $g=1$, yielding $\Aut(U,E)=D_{2r+1}$ and hence the claim.

\medskip

Now suppose that
 $\sigma:U\to V$ is an isomorphism between finite subhypergraphs of $N_3$. Since $\sim_U$ is defined from $E$, $\sigma$ maps $\sim_U$-classes to $\sim_V$-classes. Let $U_1,\ldots,U_n$ be the $\sim_U$-classes of $U$ listed in the clockwise cyclic ordering, and for each $i$ let $a_i\in U_i$, let $b_i=\sigma(a_i)$, let $V_i$ be the $\sim_V$-class of $b_i$, and put $A=\{a_1,\ldots,a_n\}$, and $B=\{b_1,\ldots,b_n\}$.
Then $A$ and $B$ are balanced, and by the claim $\sigma$ induces an isomorphism $(A,R)\to (B,R)$. We shall say that $\sigma$ is {\em positive} if it preserves the positive circular orientation on $A$, and {\em negative} otherwise (one of these holds, by the claim). 

Since $\sigma$ maps $\sim_U$-classes to $\sim_V$-classes, it follows that for each $i=1,\ldots,n$ the equivalence classes $U_i$ and $V_i$  have the same size $t_i$. For each $i=1,\ldots,n$ write $U_i=\{a_{i1},\ldots,a_{it_i}\}$ with $a_{ij}\to a_{ik}$ whenever $j<k$.  For each $i=1,\ldots,n$ we also put $V_i=\{b_{i1},\ldots,b_{it_i}\}$, where if $\alpha$ is positive we have $b_{ij}\to b_{ik}$ whenever $j<k$, and if $\sigma$ is negative we have $b_{ij}\to b_{ik}$ whenever $j>k$. 

Finally, define $\alpha:U\to V$ by putting $\sigma(a_{ij})=b_{ij}$ for each $i=1,\ldots,n$ and $j=1,\ldots,t_i$. It suffices to prove the following claim.

{\em Claim 2.} The map $\alpha:(U,R)\to (V,R)$ is an isomorphism.

{\em Proof of Claim.} For triples within a $\sim_U$-class, $\sigma$ preserves $R$ since it preserves or reverses the ordering given by $\to$. For triples meeting three distinct $\sim_U$-classes, $\alpha$ preserves $R$ as $\sigma|_A$ does, and as elements outside a $\sim$-class are $\to$-related to all elements of the $\sim_U$-class in the same way. For triples containing two elements $a_{ik}, a_{il}$ ($k<l$) from one class $U_i$, and one element $a_{jm}$ from another class $U_j$, suppose first that $\sigma$ is positive, and that $a_i \to a_j$. Then $a_{ik}, a_{il}\in a_{jm}^-$, and $a_{ik} \to a_{il}$, so $R(a_{il};a_{ik}, a_{jm})$. We have $b_{ik}\to b_{il}$ and $b_{ik}, b_{il}\in b_{jm}^-$, so $R(b_{il}; b_{ik}, b_{jm})$ as required. The other cases (where $a_j \to a_i$, and where $\sigma$ is negative) are similar. 

Given Claim 2, it follows by homogeneity of $(Z,R)$ (see Lemma~\ref{localorder})  that $\alpha$ is induced by some $g\in \Aut(Z,R)=\Aut(Z,E)$, and we have $U^g=V$ as required.

\end{proof}

{\em Proof of Theorem~\ref{3hyper}(ii).} See Proposition~\ref{N3}. \hfill $\Box$

\begin{remark} \label{two-graph} \rm 
The example $N_3$ is a {\em two-graph}, namely a 3-hypergraph with the property that any four vertices carry an even number of hypergraph edges (in this case, 2 or 4). This notion was introduced by D.G. Higman -- see \cite{evans} or \cite{cameron-book} for background (including the infinite case).
\end{remark}

\section{The case $k\geq 4$.}
In this section we first apply methods from Section 3, in particular Lemma~\ref{5verts} and \ref{PkJ}, to prove Theorem~\ref{4hyper}(i). Then in Section 4.2 we prove Theorem~\ref{4hyper}(ii) and (iii), basing our construction on the (unordered) 2-regular countable dense proper $C$-set. Theorem~\ref{6hyper} is proved in Section 4.3, exploiting a 3-branching $D$-set. 
\subsection{Proof of Theorem~\ref{4hyper}(i)}

First, observe that  if $(M,E)$ is an infinite set-homogeneous $k$-hypergraph with $k\geq 4$ whose automorphism group $G$ is not 2-transitive, then $G$ is 3-homogeneous, so  by Theorem~\ref{4hom}(i), $G$
preserves a linear order $<$ on $M$.

We now consider the case where $k=4$. We have not tried hard to apply the methods for larger $k$. 

{\em Proof of Theorem~\ref{4hyper}(i).}
Let $(M,E)$ be a ${\leq}5$-set-homogeneous countably infinite 4-hypergraph whose automorphism group $G$ is not 2-transitive. By 
 Theorem~\ref{4hom}(i)  there is a $G$-invariant dense total order $<$ on $M$, and we adopt the notation $S_{ij}$ and $P^k_J$ from Section 2.2.  By Lemmas ~\ref{5verts} and \ref{PkJ}, each formula $S_{ij}$ with $|i-j|\geq 2$ is realised by a 5-element substructure of $M$, and for each $i=1,2,3,4$, $M$ contains a 5-tuple realising $T_i$. There are 12 such $S_{ij}$, namely, $S_{13}, S_{14}, S_{15},S_{24}, S_{25}, S_{35}$ (where $i<j$) and the corresponding formulas with $i>j$ namely $S_{31}, S_{41}, S_{51}, S_{42}, S_{52}, S_{53}$. By Lemma~\ref{PkJ}, the ways in which  conditions $P^k_J$ ensure that the formulas $S_{ij}$ and $T_i$ are realised are determined by Table 2. 

\tikzset{ 
    table/.style={
        matrix of nodes,
        row sep=-\pgflinewidth,
        column sep=-\pgflinewidth,
        nodes={
            rectangle,
            draw=black,
            align=center
        },
        minimum height=1.5	em,
        text depth=0.5ex,
        text height=1.5ex,
        nodes in empty cells,
        every even column/.style={
            nodes={fill=gray!30}
        },
        column 1/.style={
            nodes={text width=2em,font=\bfseries}
        },
        }
    }

\begin{table}
\centering

\begin{tikzpicture}

\matrix [table, text width=1.6em]
{
 & $T_{1}$ & $T_{2}$ & $T_{3}$ & $T_{4}$ & $S_{13}$ & $S_{31}$ & $S_{14}$ & $S_{41}$ & $S_{15}$ & $S_{51}$ & $S_{24}$ & $S_{42}$ &
 $S_{25}$ & $S_{52}$ & $S_{35}$ & $S_{53}$ \\
$P^{1}_{1234}$  &  &  &  & 0 & & & & & & 0 & & & & 0 & & 0 \\
$P^{1}_{1245}$   &  & 0 & 0 &  &  & 0 & & & & & & & & & 0 &  \\
$P^{1}_{1235}$   &  &  & 0 & 0 &  & & & 0 & & & & 0 & & & &  \\
$P^{1}_{2345}$  & 0 &  &  &  & 0  & & 0 & & 0 & & & & & & &  \\
$P^{1}_{1345}$  & 0 & 0 &  &  &  & & & &  & & 0 & & 0 & & &  \\
$P^{2}_{123}$  &  &  & 0 & & & &  & 0 & & 0 & & 0 & & 0 & & 0 \\
$P^{2}_{124}$  &  & 0 & 0 & 0 &  & 0 & & & & 0 & & & & 0 & &  \\
$P^{2}_{125}$  &  & 0 &  & 0 &  & 0 & & 0 &  & &  & 0 & & & 0 &  \\
$P^{2}_{134}$  & 0 & 0 &  & 0 &  & & &  & & 0 & 0 & & & & & 0  \\
$P^{2}_{135}$  & 0 & 0 & 0 & 0 &  &  & & 0 & & &  & & 0 & & &  \\
$P^{2}_{145}$  & 0 &  & 0 &  &  & 0 & & & & & 0 & & 0 & & 0 &  \\
$P^{2}_{234}$  & 0 &  &  & 0 & 0 & & 0 & & & & & & & 0 & & 0 \\
$P^{2}_{235}$  & 0 &  & 0 & 0 & 0 & & & & 0 & & & 0 & & & &  \\
$P^{2}_{345}$  &  & 0 &  &  & 0 & & 0 & & 0 & & 0 & & 0 & & &  \\
$P^{2}_{245}$  & 0 & 0 & 0 &  &  &  & 0 &  & 0 & & & & & & 0 &  \\
$P^{3}_{12}$  & & 0 &  &  &  & 0 & & 0 & & 0 &  & 0 &  & 0 &  &  \\
$P^{3}_{13}$  & 0 & 0 & 0 & &  & &  & 0 & & 0 & & & &  & & 0 \\
$P^{3}_{14}$  & 0 &  & 0 & 0 &  & 0 & & & & 0 & 0 &  & & & &  \\
$P^{3}_{15}$  & 0  &  &  & 0 &  & 0 & & 0 &  & &  & & 0 & & 0 &  \\
$P^{3}_{23}$  & 0 &  & 0 &  & 0 &  &  &  & & & & 0 &  & 0 & & 0 \\
$P^{3}_{24}$  & 0 & 0 & 0 & 0 &  &  & 0 & & & &  & &  & 0 &  &  \\
$P^{3}_{25}$  & 0 & 0 &  & 0 &  & &  & & 0 & & & 0 & &  & 0 &  \\
$P^{3}_{34}$  &  & 0 &  & 0 & 0 & & 0 & &  & & 0 &  & & & & 0 \\
$P^{3}_{35}$  &  & 0 & 0 & 0 & 0 & &  & & 0 & &  & & 0 & & &  \\
$P^{3}_{45}$  &  &  & 0 &  &  &  & 0 &  & 0 & & 0 & & 0 & & 0 &  \\
$P^{4}_{1}$  & 0 &  &  &  &  & 0 & & 0 & & 0 &  & &  & &  &  \\
$P^{4}_{2}$  & 0 & 0 &  &  &  & &  & & & & & 0 & & 0 & &  \\
$P^{4}_{3}$  &  & 0 & 0 &  & 0 & & & &  & & & & & & & 0 \\
$P^{4}_{4}$  &  &  & 0 & 0 &  & & 0 & &  & & 0 & &  & & &  \\
$P^{4}_{5}$  &  &  &  & 0 &  &  &  &  & 0 & & & & 0 & & 0 &  \\
};

\end{tikzpicture}
\caption{}
\end{table}

By Lemma~\ref{PkJ}, we must show that there do not exist sets $J_1, J_2,J_3,J_4$ so that  if $P^i_{J_i}$ hold for each $i=1,\ldots,4$ then all formulas $S_{ij}$ and $T_1,T_2,T_3,T_4$ are realised. Since $|J_1|=4$ and $|J_2|=3$ and they are both subsets of $\{1,2,3,4,5\}$, $|J_1\cap J_2|\geq 2$. We consider all possible 2-sets which could lie in the intersection, using symmetry (essentially, reversing the order) to reduce the number of cases. As a small abuse, we shall write $ijK$ for the set $J=\{i,j\}\cup K$.

\begin{enumerate}
\item $M$ realises $P^1_{12J}$ and $P^2_{12K}$. These do not ensure realisation of $T_1, S_{13}, S_{14},S_{15},S_{24}$ and $S_{25}$ (see Table 2). So the latter must arise from  $P^3_J$ and $P^4_J$. The only cases which realise $S_{13}$  are the  following cases.
 \begin{enumerate}
\item $P^4_3$ which needs $T_1, S_{14}, S_{15}$ from some $P^3_J$, which cannot occur.
 
\item $P^3_{23}$ which needs $S_{14}, S_{15}$ from $P^4_J$,  impossible.
 
\item $P^3_{34}$ which needs $T_{1}, S_{25}$   from $P^4_J$, impossible.
 
\item  $P^3_{35}$ which needs $S_{14}, T_1$ from $P^4_J$, again impossible.

\end{enumerate}

\item $M$ realises $P^1_{13J}$ and $P^2_{13K}$.  These do not realize $S_{13}, S_{31}, S_{14}, S_{15}, S_{35}$ which need to be realized by $P^3$ and $P^4$. The only possibilities of $S_{13}$ are:
 \begin{enumerate}
\item $ P^4_3$,  which still needs to realize $S_{31}, S_{14}$ from some $P^3_J$, which cannot occur.
 
\item $P^3_{23}$ which needs $S_{31}, S_{14}$ from $P^4_J$, impossible.
 
\item $P^3_{34}$ which needs $S_{31}, S_{35}$   from $P^4_J$, impossible.
 
\item $ P^3_{35}$ which needs $S_{31}, S_{35}$ from $P^4_J$, which cannot happen.
\end{enumerate}

\item $M$ realises $P^1_{14J}$ and $P^2_{14K}$. This is almost the same as the last case. These conditions  do not realize $S_{13}, S_{14}, S_{41}, S_{15}, S_{42}$ which need to be realized by $P^3 $ and $P^4$. As before, the only possibilities of $S_{13}$ are: 
\begin{enumerate} 
\item $P^4_3$ which needs $ S_{41}, S_{14}$ from some $P^3_J$, impossible.
 
\item $P^3_{23}$, needs $S_{41}, S_{14}$ from $P^4_J$, impossible.
 
\item $P^3_{34}$, needs $S_{41}, S_{15}$   from $P^4_J$, impossible.
 
\item  $P^3_{35}$,  needs $S_{41}, S_{14}$ from $P^4_J$, again impossible.
\end{enumerate}

\item $M$ realises $P^1_{15J}$ and $P^2_{15K}$. These do not realise $ S_{13}, S_{14}, S_{15}, S_{51}, S_{52},S_{53}$, which need to be realized by $P^3$ and $P^4$. Again, the only possibilities for $S_{13}$ are:
$P^4_3, P^3_{23}, P^3_{34}, P^3_{35}$. The first three need both $S_{15}$ and $S_{51}$ from the remaining condition, which is clearly impossible.  $P^3_{35}$ needs $S_{14}, S_{51}$ from $P^4_J $, again impossible.

\item $M$ realises $P^1_{23J}$ and $P^2_{23K}$. These do not realize $S_{31}, S_{24}, S_{25}, S_{35},T_2$  which need to be realized by $ P^3$ and $P^4$. The only possibilities of  $S_{31}$ are:
 \begin{enumerate}
\item $P^4_1$, which  needs $S_{24}, S_{25},T_2$ from some $ P^3_J$, impossible.
 
\item  $P^3_{12}$, needs $ S_{24}, S_{25}$ from $P^4_J$, impossible.
 
\item $P^3_{14}$,  needs $S_{25}, S_{35},T_2$   from $P^4_J$, impossible.
 
\item $P^3_{15}$,  needs  $S_{24},T_2$ from $P^4_J$.
\end{enumerate}

\item $M$ realises $P^1_{24J}$ and $P^2_{24K}$. These do not realize $S_{41}, S_{24}, S_{42}, S_{25}$,  which need to be realized by $P^3$ and $P^4$. The only possibilities of $S_{41}$ are:
 \begin{enumerate}
\item  $P^4_1$, needs $S_{42}, S_{24}$ from some $P^3_J$, impossible.
 
\item $P^3_{12}$, needs $S_{24}, S_{25}$ from $P^4_J$, impossible.
 
\item $P^3_{13}$,  needs $S_{24}, S_{25}$   from $P^4_J$.
 
\item $P^3_{15}$,  needs $S_{42}, S_{24}$ from $P^4_J$, impossible.
\end{enumerate}

\end{enumerate}

The remaining cases follow by symmetry from the above, arguing with the order reversed. For example, the argument in Case (I) also eliminates that where $M$ realises $P^1_{45J}$ and $P^2_{45K}$. \hfill $\Box$

\subsection{Set-homogeneous 4-hypergraphs with 2-transitive not 2-primitive  automorphism group}
We here prove Theorem~\ref{4hyper}(ii), a consequence of the following result. In this subsection $(M,C)$ denotes the 3-branching dense proper  $C$-set as defined in Section 2.1 (so the reduct of the structure  $(M,C,\leq)$  from Theorem~\ref{2-reg}).

\begin{proposition}\label{ex4hyp}
 Define a 4-hypergraph structure on $M$ whose edge set $E$ consists of 4-sets of form $\{x_1,x_2,y_1,y_2\}$ such that $C(x_i;y_1,y_2)$ and $C(y_i;x_1,x_2)$ hold for $i=1,2$. 
Then $M_4=(M,E)$ is set-homogeneous and has 2-transitive but not 2-primitive automorphism group, so is not a homogeneous 4-hypergraph.
\end{proposition}

\begin{proof} As noted in Section 2.1,  $(M,C)$ is a homogeneous structure. The 4-sets prescribed to form edges of $M_4$ are those as in Figure 6.

\begin{figure}[H]
	\centering
	\begin{tikzpicture}
	
	\draw (7,0)--(5,2)node[above]{$x_{1}$}; \draw (5.6,1.4)--(6.7,2)node[above]{$x_{2}$};
	\draw (7,0)--(9,2)node[above]{$y_{1}$}; \draw (8.5,1.5)--(7.5,2)node[above]{$y_{2}$};
	
	\end{tikzpicture}
	\caption{}
	
\end{figure}

Let $G=\Aut(M_4)$. We observe first that $G$ preserves the relation $C$. Indeed, $C$ is $\emptyset$-definable in $(M,E)$: 
for  $x,y,z\in M$  we have that $C(x;y,z)$ holds if and only if 
$$\big( y=z\wedge x \neq y\big) \vee \big( \exists u\exists v( x,y,z,u,v\mbox{~are distinct and has only the   edge~} yuvz)\big).$$ See Figure 7 for a configuration indicating the left-to-right direction here.
Thus, $\Aut(M_4)=\Aut(M,C)$. (This can also be proved via \cite[Corollary 2.3]{bodirsky}, as done in the proof of Claim 1 in Proposition~\ref{6hyp} below.) The  group $\Aut(M,C)$ is 2-transitive (by homogeneity of $(M,C)$) but not 2-primitive -- for $a\in M$ the stabiliser $G_a$ preserves a proper non-trivial equivalence relation $\sim$ on $M\setminus \{a\}$, where $x \sim y\Leftrightarrow C(a;x,y)$. Thus, $\Aut(M_4)$ is 2-transitive but not 2-primitive. 
\begin{figure}[H]
	\centering
	\begin{tikzpicture}
	
	\draw (0,0)--(4,3)node[above]{$z$};
	\draw (0,0)--(-0.9,2)node[above]{$x$};
	\draw (3.3,2.45)--(2.7,3.5)node[above]{$v$};
	
	\draw (1.6,1.2)--(0.5,3.6)node[above]{$y$};
	\draw (0.9,2.7)--(1.6,3.6)node[right]{$u$};
	
	\end{tikzpicture}
	\caption{}
	
\end{figure}

We show next that $M_4$ is set-homogeneous. So suppose $U,V$ are finite subsets of $M$ which carry isomorphic induced 4-hypergraphs, with $\sigma:(U,E) \to (V,E)$ an isomorphism.  We show by induction on $|U|$  that $(U,C)\cong (V,C)$, from which it follows by homogeneity of $(M,C)$  that there is $g\in \Aut(M,C)=\Aut(M_4)$ with $U^g=V$, as required. We may assume that $|U|\geq 5$, essentially as $\Aut(M,C)$ has just two orbits on 4-sets. 

By a {\em null} hypergraph we mean one none of whose 4-subsets are edges. It is easily seen that if $W$ is a null (induced) subhypergraph of $(M,E)$ of size $n$, then  we may write $W=\{w_1,\ldots,w_n\}$ so that $C(w_i;w_j,w_k)$ whenever $i<j<k$.  We may assume that $U$ and $V$ are not null, since otherwise they carry isomorphic $C$-structures. 

Choose a null subhypergraph $A$ of $U$ of maximal size $n$ say, and let $B=\sigma(A)$ (so $B$ is a null subhypergraph of $V$ of maximal size). As noted above, we may write $A=\{a_1,\ldots,a_n\}$ and $B=\{b_1,\ldots,b_n\}$ so that $C(a_i;a_j,a_k)$ and $C(b_i;b_j,b_k)$ whenever $i<j<k$ (but we are {\em not} assuming that $\sigma(a_i)=b_i$ for each $i$).
 We may assume $n\geq 4$, since otherwise $|U|\leq 4$, which is easily handled.


We say that $u\in  U\setminus A$ is {\em high} if $A \cup \{u\}$ has a single edge. By the maximality of $|A|$, if there is high $u\in U\setminus A$ then $u$ is uniquely determined, and $C(a_n;u,a_{n-2})$ holds (the edge is $ua_{n-2}a_{n-1}a_n$). Furthermore in this case, $v:=\sigma(u)$ is high (in the corresponding sense) in $V$, is uniquely determined by this property, and $C(b_n;v,b_{n-2})$, and also $\sigma(\{a_{n-2},a_{n-1},a_n\})=\{b_{n-2},b_{n-1},b_n\}$. If such $u,v$ exist, put $A^*:=\{u,a_{n-2},a_{n-1},a_n\}$ and $B^*:=\sigma(A^*)=\{v,b_{n-2},b_{n-1},b_n\}$, and otherwise put $A^*=B^*=\emptyset$.

For each $a\in A\setminus A^*$ put $S(a)=\{x\in U: (A\setminus \{a\})\cup \{x\} \mbox{~is null of size~}n\}$ and for $b\in B\setminus B^*$ define $S(b)$ similarly.  It can be checked that the sets $S(a)$ partition $U\setminus A^*$ and the sets $S(b)$  partition $V\setminus B^*$, and that $\alpha$ respects these partitions. We shall say that $a\in A\setminus A^*$ is {\em isolated} if $|S(a)|=1$, and similarly for $b\in B\setminus B^*$. Let
$I(A)=\{a\in A: a\mbox{~is isolated}\}$, and define $I(B)$ similarly. Thus, $\sigma(I(A))=I(B)$.

If $a\in A\setminus A^*$ is non-isolated, we shall say that $S(a)$ is {\em low} if for any distinct elements $x,y\in S(a)$ and distinct $z,w\in A\setminus \{a\}$,  $xyzw$ is an edge (with a similar definition for $S(b)$ with $b\in B\setminus B^*$ non-isolated). By considering the $C$-relation, it can be checked that $S(a)$ is low if and only if $a=a_1$ and $a_1$ is non-isolated. Since $\sigma$ respects lowness, it follows that  $S(a_1)$ is low if and only if $S(b_1)$ is low, and that if this holds then $\sigma(S(a_1))=S(b_1)$ so they carry isomorphic hypergraph structures, and likewise $\sigma$ induces an induced hypergraph  isomorphism $U\setminus S(a_1) \to V\setminus S(b_1)$. By induction it follows that the   $C$-sets induced on $S(a_1)$ and $S(b_1)$ are isomorphic, as are the $C$-sets  induced on $U\setminus S(a_1)$ and $V\setminus S(b_1)$, and hence that the  $C$-sets induced on $U$ and $V$ are isomorphic, as required.

Thus, we may suppose that $a_1$ and $b_1$ are both isolated. Let $r$ be maximal so that $a_1,\ldots,a_r$ are all isolated. 
We shall say that an isolated element $a\in A$ is {\em special}, if for any  $a'\in I(A)\cup A^*$, non-isolated $c\in A\setminus A^*$, and distinct $c_1,c_2\in S(c)$, the set $aa'c_1c_2$ is a non-edge. It can be checked that the special elements of $A$ are exactly $a_1,\ldots,a_r$. See Figure 8 for the intended configuration in terms of the $C$-relation. Since $\sigma$ respects specialness, the special elements of $B$ are $b_1,\ldots,b_r$, and $\sigma(\{a_1,\ldots,a_r\})=\{b_1,\ldots,b_r\}$. Thus, $\sigma$ induces a hypergraph isomorphism $U\setminus\{a_1,\ldots,a_r\} \to V\setminus \{b_1,\ldots,b_r\}$. Since these sets are smaller than $|U|$, it follows by induction that they carry isomorphic  $C$-set structures, and hence that $U$ and $V$ carry isomorphic  $C$-set structures, as required.

\begin{figure}[H]
	\centering
	\begin{tikzpicture}
	\tikzstyle{every node}=[font=\tiny]
	\draw (0,0)--(4,3)node[above]{$a_{n}$};
	\draw (0,0)--(-.2,1.7)node[above]{$a_{1}$};
	\draw [decorate,decoration={brace,amplitude=7pt,raise=-10pt},yshift=0pt]
	(1.1,3.5) -- (2,3.9) node [black,midway,xshift=-0.3cm, yshift=0.4cm]  {\footnotesize $S({a_{r+1}})$};
	
	\draw (.5,.4)--(.3,2.1)node[above]{$a_{2}$};
	\draw (1.4,1.05)--(1.2,2.6)node[above]{$a_{r}$};
	\draw[dashed] (.5,1.4)--(1.3,2);
	\draw (3,2.25)--(2.8,3.3)node[above]{$a_{n-2}$};
	\draw (2,1.5)--(1.8,3)node[above]{$a_{r+1}$};
	\draw (1.95, 2)--(2.1,2.5);
	\draw (1.9,2.2)--(1.6,2.6);
	\draw[dashed] (2,2)--(2.95,2.6);
	\draw (3.5,2.6)--(3.3,3.5)node[above]{$a_{n-1}$};
	
	\end{tikzpicture}
	\caption{}
\end{figure}



\end{proof}

The following result now gives us Theorem~\ref{4hyper}(iii).

\begin{corollary} \label{linbet} If $(N,E)$ is a ${\leq}5$-set-homogeneous countably infinite 4-hypergraph whose automorphism group $K$  is not 2-primitive, then either $(N,E)$ is isomorphic to the structure $M_4$ of Proposition~\ref{ex4hyp}  or its complement, or $K$ preserves a linear betweenness relation on $N$. 
\end{corollary}

\begin{proof} The set-homogeneity assumption ensures that $K$ is 3-homogeneous on $N$. Hence, by Theorem~\ref{4hyper}(i), $K$ is 2-transitive on $N$. Hence, by Theorem~\ref{4hom}(ii), $K$ preserves on $N$ a linear betweenness relation or a relation $C$ on $N$ so that $(N,C)$ is isomorphic to the structure $(M,C)$ from Theorem~\ref{2-reg}. 

So suppose $K\leq \Aut(N,C)$. Since $\Aut(N,C)$ has two orbits on 4-sets and $(N,E)$ is ${\leq}5$-set-homogeneous, $K$ also has the same two orbits on 4-sets of $N$. One such orbit gives a hypergraph isomorphic to $M_4$, and the other gives its complement. 


\end{proof}

We consider a further 4-hypergraph. Let  $(N,D)$ be the countably infinite dense proper homogeneous 4-branching $D$-set (see Section 2.1). Define a 4-hypergraph $N_4=(N,E)$ from $(N,D)$, whose edges are 4-sets which satisfy $D$ under some ordering.

\begin{proposition}
The 4-hypergraph $N_4$ is ${\leq}5$-set-homogeneous but not 6-set-homogeneous.
\end{proposition}

\begin{proof} First, we observe that the $D$-relation can be defined from the hypergraph structure. Given distinct $x,y,z,w\in N$, put $D(x,y;z,w)$ if and only if
 there is $ u\in N\setminus \{x,y,z,w\}$ such that $\{x,y,z,w,u\}$ has only the non-edges $xyuz, xyuw$.
It can be checked that this correctly recovers the $D$-relation from $E$, so $\Aut(N_4)=\Aut(N,D)$.

Homogeneity of $(N,D)$ and  inspection of possible 5-element substructures easily yields that $N_4$ is ${\leq}5$-set-homogeneous. It is not 6-set-homogeneous, as $N_4$ has 6-vertex complete subhypergraphs corresponding to the non-isomorphic  $D$-sets in Figure 9.
\begin{figure}[H]
	\centering
	\begin{tikzpicture}
	\draw (4,0)--(6,0);
	\draw (4,0)--(3,1);
	\draw (4,0)--(3,-1);
	\draw (6,0)--(7,1);
	\draw (6,0)--(7,-1);
	
	\draw (4.5,0)--(4.5,1);
	\draw (5.5,0)--(5.5,1);
	
	\draw (9,0)--(11,0);
	\draw (9,0)--(8,1);
	\draw (9,0)--(8,-1);
	\draw (11,0)--(12,1);
	\draw (11,0)--(12,-1);
	
	\draw (10,0)--(10,1);
	\draw (9.5,1.5)--(10,1)--(10.5,1.5);

	\end{tikzpicture}
	\caption{}
	
\end{figure}
\end{proof}

\subsection{Set-homogeneous 6-hypergraphs}

In this section we prove Theorem~\ref{6hyper}. 
The main point is the existence assertion (i), restated below. The example is defined from a $D$-set.

\begin{proposition}\label{6hyp}
There is a countably infinite set-homogeneous 6-hypergraph $M_6=(M,E)$ whose automorphism group $G$ is 3-transitive but not 3-primitive.
\end{proposition}

\begin{proof} Our starting point is the group $J$ described in Theorem 5.1 of \cite{cameron4}. The language in that paper is different, but  $J$ is the automorphism group of the unique countable homogeneous $D$-set $(M,D)$ with branching number 3, as described in Section 2.1 above and in Section 32 of \cite{AN1}.  It is easily checked and noted in \cite[Theorem 5.1]{cameron4}) that $J$ is 5-homogeneous but has two orbits on the collection of subsets of size 6. We define a 6-hypergraph $M_6$ on $M$ in which the edges are sets whose induced $D$-structure has the isomorphism type in Figure 10. Let $E$ denote the resulting 6-ary edge-relation on $M$. Clearly $J=\Aut(M,D)\leq \Aut(M_6)$.

\begin{figure}[H]
	\centering
	\begin{tikzpicture}
	\draw (4,0)--(6,0);
	\draw (4,0)--(3,1)node[above]{$x_{2}$};
	\draw (4,0)--(3,-1)node[below]{$x_{1}$};
	\draw (6,0)--(7,1)node[above]{$x_{5}$};
	\draw (6,0)--(7,-1)node[below]{$x_{6}$};
	
	\draw (4.5,0)--(4.5,1)node[above]{$x_{3}$};
	\draw (5.5,0)--(5.5,1)node[above]{$x_{4}$};
	
	\end{tikzpicture}
	\caption{}
	
\end{figure}

{\em Claim 1.}  $\Aut(M_6)=J$.

{\em Proof of Claim.} We must define $D$ from $E$.
It follows from \cite[Corollary 2.3]{bodirsky} that the structure $(M,D)$ has no first-order reducts (up to interdefinability over $\emptyset$) other than itself and $(M,=)$; that is, any structure on $M$ which is $\emptyset$-definable in $(M,D)$ is interdefinable over $\emptyset$ with $(M,D)$ or $(M,=)$. Since $(M,E)$ is $\emptyset$-definable in $(M,D)$ and is not a complete or null hypergraph, it is interderdefinable with $(M,D)$, giving the claim. (Formally, \cite[Corollary 2.3]{bodirsky} describes the reducts of the structure
$(M,C)$ from Section 4.2 above, and shows that the only proper non-trivial reduct is $(M,D)$, which is in \cite{bodirsky} denoted $(\mathbb{L},Q)$).

\medskip

To show that $M_6$ is set-homogeneous, using the homogeneity of $(M,D)$ it suffices to prove the following claim.

{\em Claim 2.} If $U,V\subset M$ are finite and $\sigma:(U,E) \to (V,E)$ is an isomorphism of the induced subhypergraphs, then the structures $(U,D)$ and $(V,D)$ are isomorphic. 

{\em Proof of Claim.} Let $A$ be a complete subhypergraph of $U$ of maximal size, and put $B=\sigma(A)$. It is easily seen that the $D$-structure on $A=\{a_1,\ldots,a_{n+2}\}$ has the form depicted in Figure 11. Likewise $B=\{b_1,\ldots,b_{n+2}\}$ carries a $D$-structure as depicted. We do {\em not} claim that $\sigma(a_i)=b_i$ for each $i$. 

For each $i=3,\ldots,n$, let $T_i=\{e: a_1a_2ea_ia_{n+1}a_{n+2} \mbox{~is a non-edge}\}$. Clearly $T_i$ is as depicted in Fig. 11. By maximality of $A$, we have $U\setminus A=\bigcup(T_i:3\leq i\leq n)$, and maximality of $|A|$ yields that $|T_3|\leq 1$ and $|T_n|\leq 1$.

\begin{figure}[H]
	\centering
	\begin{tikzpicture}
	\draw (4,0)--(8,0);
	\draw (4,0)--(3,1)node[above]{$a_{1}$};
	\draw (4,0)--(3,-1)node[below]{$a_{2}$};
	\draw (8,0)--(9,1)node[above]{$a_{n+1}$};
	\draw (8,0)--(9,-1)node[below]{$a_{n+2}$};
	
	\draw (4.5,0)--(4.5,1)node[above]{$a_{3}$};
	\draw (5.3,0)--(5.3,1);
	
	\draw (6,0)--(6,2.7)node[above]{$a_{i}$};
	\draw (7.5,0)--(7.5,1)node[above]{$a_{n}$};
	\draw (6.7,0)--(6.7,1);
	
	\draw (6,1.2)--(5.5,1.6);
	\draw (6,2)--(6.3,2.6);
	\draw (6.15,2.3)--(6.5,2.5);
	
	\draw [decorate,decoration={brace,amplitude=5pt,mirror,raise=2pt},yshift=2pt]
	(6.5,3) -- (5.6,3) node [black,midway,xshift=0cm, yshift=.8cm] {\footnotesize $T_{i}\cup\{a_{i}\}$};
	
	\end{tikzpicture}
	\caption{}
	
\end{figure}

For each $e\in U\setminus A$, write $e\sim a_i$ if $(A\setminus \{a_i\})\cup \{e\}$ carries a complete hypergraph. It can be checked that  if $e\in T_3$ then $e\sim a_i$ for each $i\in \{1,2,3\}$ and if $e\in T_n$ then $e\sim a_i$ for $i\in \{n,n+1,n+2\}$ but if $e\in T_i$ where $4\leq i \leq n-1$ then $\{j: e\sim a_j\}=\{i\}$. We say that $e\in U\setminus A$ is {\em peripheral} if $|\{i: e \sim a_i\}|=3$, and that $e$ is {\em central} otherwise. We define $T_i'$ correspondingly in $V$ (with each $b_i$ replacing $a_i$) and {\em central} and {\em peripheral} in the same way. Let $P, P'$ be the sets of peripheral vertices of $U,V$ respectively. Clearly $|P|\leq 2$ and $\sigma(P)=P'$, and $\sigma$ induces some permutation $\pi$ of $\{4,\ldots,n-1\}$ such that if $e$ is central in $U$ then $e\in T_i$ if and only if $\sigma(e)\in T_{\pi(i)}'$. 

Relabelling the $b_i$ if necessary (via a map $i\mapsto n+2-i$) we may suppose that $\sigma(T_3)=T_3'$. Let $R_3=T_3\cup\{a_1,a_2,a_3\}$ and $R_3'=T_3'\cup\{b_1,b_2,b_3\}$. 
Then $|R_3|=|R_3'|\leq 4$.


Consider now the  $C$-relation $C_U$ induced on $U\setminus R_3$ with downwards direction towards $a_1$; for $x,y,z\in U\setminus R_3$ we have $C_U(x;y,z)\Leftrightarrow D(a_1,x;y,z)$. Similarly let $C_V$ be the  $C$-relation  induced on $V\setminus R_3'$ with downwards direction towards $b_1$. Observe that for any distinct $u_1,u_2,u_3,u_4\in U\setminus R_3$, and any distinct $x,y\in R_3$ and distinct $x',y'\in R_3$, $xyu_1u_2u_3u_4$   is a hypergraph edge if and only if $x'y'u_1u_2u_3u_4$ is a hypergraph edge. Thus, a canonical 4-hypergraph structure is induced on $U\setminus R_3$ and its complement is derived from the relation $C_U$ on this set as in Proposition~\ref{ex4hyp}. Likewise, a canonical 4-hypergraph is induced on $V\setminus R_3'$, and is derived from $C_V$. Since $\sigma$ induces an isomorphism of the 4-hypergraph on $U\setminus R_3$ onto that on $V\setminus R_3'$, it follows by the proof of Proposition~\ref{ex4hyp} that the corresponding structures $(U\setminus R_3, C_U)$ and $(V\setminus R_3', C_V)$ are isomorphic, and hence that $U$ and $V$ carry isomorphic $D$-substructures, yielding Claim 2 and hence the result. 
\end{proof}

\begin{corollary}\label{6hypsep}  Let $(N,E)$ be a set-homogeneous 6-hypergraph whose automorphism group is not 3-primitive. Then either $(N,E)$ or its complement is isomorphic to the structure $M_6$ from Proposition~\ref{6hyp}, or $\Aut(N,E)$ preserves a separation relation on $N$.
\end{corollary}

\begin{proof}
Let $K=\Aut(N,E)$. Then $K$ is a 5-homogeneous permutation group which is not 3-primitive. If $K$ is not 3-transitive on $N$, then by Theorem~\ref{4hom}(iii) $K$ preserves  a circular order or linear betweenness relation, and hence preserves a separation relation on $N$, as the latter is definable without parameters in a circular order or linear betweenness relation. So we may suppose that $K$ is 3-transitive. It follows by Theorem~\ref{4hom}(iv) that, assuming $K$ does not preserve a separation relation on $N$, then there is a $K$-invariant $D$-relation $D$ on $N$ such that $(N,D)$ is isomorphic to the 3-branching $D$-set $(M,D)$ from the proof of Proposition~\ref{6hyp}. Since $\Aut(N,D)$ has two orbits on 6-sets,  the result follows (one orbit gives a hypergraph isomorphic to $M_6$, and the other gives the complement).
\end{proof}

{\em Proof of Theorem~\ref{6hyper}.} See  Proposition~\ref{6hyp} and Corollary~\ref{6hypsep}. \hfill $\Box$

For completeness we also record the following immediate corollary of Theorem~\ref{4hom}(iii), which may make applicable methods similar to those of Lemma~\ref{PkJ}. Part (v) of Theorem~\ref{4hom}  also has consequences for set-homogeneous $k$-hypergraphs for $k\geq 9$. 
\begin{corollary}
 Let $k\geq 5$ and let $N$ be an infinite set homogeneous $k$-hypergraph whose automorphism group is not 3-transitive. Then there is an $\emptyset$-definable (so $\Aut(N,E)$-invariant)  linear betweenness relation, or circular order on $N$.

\end{corollary}

\section{Finite set-homogeneous structures, further questions}

The literature on set-homogeneity is not well-developed. In this section explore briefly what can be said about finite set-homogeneous hypergraphs, and then pose some tentative conjectures and questions, concerning how far the notion extends beyond homogeneity. 

\subsection{The finite case}

The finite homogeneous 3-hypergraphs were classified by Lachlan and Tripp in  \cite{lachlan}. There are just  four examples, with automorphism groups PGL${}_3(2)$, PGL${}_3(3)$, PSL${}_2(5)$, and the extension of PSL${}_2(9)$ by an involutory field automorphism, each in the natural action on the projective plane or line. By homogeneity, the automorphism group of any homogeneous 3-hypergraph is 2-transitive, and via the classification of finite simple groups all finite 2-transitive groups are known, making the Lachlan-Tripp result feasible. We have not attempted to carry out the corresponding classification under set-homogeneity,  but first note that such hypergraphs again have 2-transitive automorphism group by the following result. Finite $k$-homogeneous groups which are not $k$-transitive (for some $k\geq 2$) were classified by Kantor in \cite{kantor}, but this is not needed for the next result.

\begin{lemma}\label{2trans}
Let $k\geq 3$, let $\Gamma=(X,E)$ be a finite set-homogeneous $k$-hypergraph, and put $G=\Aut(\Gamma)$. Then $G$ acts 2-transitively on $X$.
\end{lemma}

 \begin{proof} This follows immediately from Lemma~\ref{linear} for $k=3$, and from Theorem~\ref{4hom}(i) (and the rigidity of finite linear orders) for larger $k$. (As stated Theorem~\ref{4hom}(i) has an assumption that $X$ is infinite, but the proof in \cite{cameron2} only uses this to ensure that $G$ is 2-homogeneous, which holds here anyway.)
\end{proof}

Enomoto \cite{enomoto} gave a very short proof that any finite set-homogeneous graph is homogenous. As noted in \cite[Lemma 3.1]{gray+}, Enomoto's argument works also for finite tournaments, but not for finite digraphs (a directed 5-cycle is set-homogeneous but not homogeneous). The following example shows that his argument (at least in the original form) is not applicable to 3-hypergraphs. 

\begin{example} \label{finiteex}\rm
We give an example of a finite set-homogeneous 3-hypergraph which is not homogeneous. Let $G$ be the group ${\rm AGL}_1(7)=(\mathbb{F}_7,+)\rtimes (\mathbb{F}^*,\cdot)$ acting on $\mathbb{F}_7=\{0,1,2,3,4,5,6\}$ with operations modulo 7. This group is 2-transitive and it can be checked with bare hands that it has 2 orbits on sets of size 3, namely an orbit $\Omega_1$ of size 14 containing the 3-sets
$$013,026,023,045,015,046,124,235,346,156,134,245,356,126$$
 (listing 3-sets as triples), and an orbit $\Omega_2$ of size 21 containing the remaining 3-sets. Let $M$ by the 3-hypergraph on vertex set $\mathbb{F}_7$ whose edges are the 3-sets in $\Omega_1$. The group $G$ also has two orbits on sets of size 4, namely $\Theta_1$ consisting of complements of sets in $\Omega_1$, and 
$\Theta_2$ containing the complements of elements of $\Omega_2$.
Observe that $\{2,4,5,6\}\in \Theta_1$, and $\{3,4,5,6\}\in \Theta_2$. The induced hypergraph on $\{3,4,5,6\}$ contains two edges ($346$ and $356$), and that on $\{2,4,5,6\}$ has 1 edge ($245$), so they are non-isomorphic.  Since $G$ is transitive on 5-sets and 6-sets from $\{0,1,\ldots,6\}$, it  follows that $G$ acts set-homogeneously on $M$.

The hypergraph $M$ is not homogeneous, since it does not occur among the examples in \cite{lachlan}. Indeed, the only 7-vertex example in \cite{lachlan} is the Fano plane, which has ${\rm PSL}_3(2)$ as automorphism group. Since the latter is simple of order 168, it cannot have $G$ as a subgroup -- indeed, $|G|=42$, and a group of order 168 with a subgroup of order 42 must have a proper normal subgroup of index at most 4!, so cannot be simple.
\end{example}

{\em Proof of Theorem D.} See Lemma~\ref{2trans}  and Example~\ref{finiteex}. \hfill $\Box$

\medskip

Lachlan (see \cite{lachlan-old} and also \cite{cherlach} for  the existence of a bound on rank) developed a very general structure theory for finite homogeneous relational structures. First, recall that a countably infinite structure $M$ is {\em smoothly approximable} by a sequence $M_0\leq M_1 \leq\ldots$ of finite substructures if $M$ is $\omega$-categorical, and (with $G=\Aut(M)$), for any $i\in \mathbb{N}$ and tuples $\bar{u},\bar{v}$ from $M_i$, $\bar{u}$ and $\bar{v}$ lie in the same $\Aut(M)$-orbit if and only if they lie in the same orbit of the setwise stabiliser $\Aut(M)_{\{M_i\}}$ of $M_i$. A rich theory around smooth approximation is developed in \cite{ch}. 

Roughly, the Lachlan theory says that if $L$ is a finite relational language, then the finite homogeneous $L$-structures consist of finitely many `sporadic' examples, and finitely many infinite families of examples, so that within each family the isomorphism type is determined by  finitely many `dimensions' taking values in $\mathbb{N}$, the dimensions varying independently and freely above a certain minimum. The infinite `limits' of these families are exactly the  homogeneous countably-infinite $L$-structures which are {\em stable} (see Section 5.2); they are `smoothly approximated' by the finite families. It can be shown that something very similar holds under set-homogeneity, and we make this precise below. 

\begin{theorem}\label{lachlanfinite}
Let $L$ be a finite relational language and let $\mathcal{C}$ be the collection of all finite set-homogeneous $L$-structures. Then we may write $\mathcal{C}= \mathcal{F}_0 \cup \mathcal{F}_1 \cup \ldots \cup\mathcal{F}_t$ where $\mathcal{F}_0$ is finite, and the structures in each $\mathcal{F}_i$ (for $i=1,\ldots,t$) smoothly approximate a set-homogeneous countably infinite $L$-structure $M_i$.
\end{theorem}

\begin{proof} As $L$ is fixed, set-homogeneity ensures that there is $d\in \mathbb{N}$ such that for each $M \in \mathcal{C}$, $\Aut(M)$ has at most $d$ orbits on $M^4$. The result now follows from Theorem 4.4.1 of \cite{wolf}, which is based on results  from \cite{ch}. It is almost immediate from the definition of smooth approximation that if a family $\mathcal{F}_i$ smoothly approximates an infinite structure $M_i$ then $M_i$ will itself be set-homogeneous.
\end{proof}

\subsection{Infinite set-homogeneous structures}

It seems feasible to classify set-homogeneous hypergraphs which are not $t$-homogeneous for some small $t$. 

\begin{problem} For $k\geq 3$, classify set-homogeneous countably infinite $(k+1)$-hypergraphs whose automorphism group is not $k$-transitive. In particular, classify set-homogeneous 4-hypergraphs whose automorphism group is not 2-primitive (see Corollary~\ref{linbet}), and set-homogeneous 6-hypergraphs whose automorphism group is not 3-primitive  (see Corollary~\ref{6hypsep}).

\end{problem}
In  the  case of Corollary~\ref{linbet} we must consider cases where there is an underlying invariant linear betweenness relation, and for Corollary~\ref{6hypsep} one considers an invariant separation relation. It seems likely that the methods of Lemmas~\ref{5verts} and \ref{PkJ} are applicable. 
As a special case of the above problem, again with these lemmas as an approach, we ask for the following extensions of Theorems ~\ref{3hyper}(i) (the uniqueness assertion) and \ref{4hyper}(i).
\begin{problem}
Show that for $k\geq 3$ the only set-homogeneous $k$-hypergraph whose automorphism group is not 2-transitive is $M_3$. Show also that $M_3, N_3$ and $N_3^c$ are the only set-homogeneous 3-hypergraphs which are not 3-homogeneous (that is, do not have $S_3$ induced on both edges and non-edges). 
\end{problem}

\begin{remark}\label{reduct} \rm
If $M,N$ are $\omega$-categorical structures with the same domain, then $N$ is said to be a (first-order) {\em reduct} of $M$ if $\Aut(N)\geq\Aut(M)$ (as subgroups of $\Sym(M)$), and to be a {\em proper} reduct if $\Aut(N)>\Aut(M)$. Thus, by the Ryll-Nardzewski Theorem,  $N$ is a  reduct of $M$ if and only if every $\emptyset$-definable relation of $N$ is $\emptyset$-definable in $M$. Observe that $M_4$ is a proper reduct of $M_3$ and $M_6$ is a proper reduct of $M_4$. Indeed, if $C$ is the dense 3-branching $C$-relation on $M_3$ and $D$ is the corresponding 3-branching $D$-relation as described in Section 2.1, we have
$\Aut(M_3)=\Aut(M,C,\leq)<\Aut(M,C)=\Aut(M_4)<\Aut(M,D)=\Aut(M_6)$ -- see the proofs of Propositions~\ref{exist}, \ref{ex4hyp}, and \ref{6hyp}. By the main theorem of \cite{bodirsky}, $M_6$ has no proper reducts other than the `trivial' one with automorphism group $\Sym(M_6)$. 
\end{remark}


It seems quite possible that for $k\geq 3$ we have found all infinite set-homogeneous but not homogeneous $k$-hypergraphs. We have checked that no other examples have  the same automorphism group as the universal homogeneous $t$-branching $C$ or $D$-relation for any $t$, and closely related structures (the most generic semilinear orders, and general betweenness relations, as developed in \cite{AN1}) also seem not to give examples. Likewise, there are no examples (other than $M_3$) obtained by expanding the above $C$-structures by a compatible total order. There is an analogous notion of compatibility of a $D$-relation with a circular order, but  this gives no examples. There are other related structures considered in \cite{cameron-tree} which we have not checked. There also appear to be obstructions to crude amalgamation arguments designed to build a homogeneous expansion of a set-homogeneous hypergraph in which we aim to force that (for the hypergraph) the group induced on some finite subset is smaller than the full automorphism group of the induced subhypergraph. 

\begin{problem} Classify {\em finite} set-homogeneous $k$-hypergraphs for $k\geq 3$. 
\end{problem}

The next question was also asked in Remark 2 of \cite[p.91]{droste+}.

\begin{problem}\label{homogenizable} Give an example of a set-homogeneous countably infinite structure over a finite relational language which is not homogenizable, that is, cannot be made homogeneous by adding symbols for finitely many $\emptyset$-definable relations.
\end{problem}

Following on from \cite{mac}, which proves the corresponding result assuming homogeneity, we pose the following. 

\begin{problem} Show that if $M$ is set-homogeneous over a finite relational language then no infinite group is interpretable in $M$. 
\end{problem}

A complete theory $T$ is {\em stable} if there do not exist a formula $ \phi(\bar{x},\bar{y})$, $M\models T$, and $\bar{a}_i \in M^{|\bar{x}|}$ and $\bar{b}_i\in M^{|\bar{y}|}$ for $i \in \omega$, such that for $i,j\in \omega$ we have $M\models \phi(\bar{a}_i,\bar{b}_j) \Leftrightarrow i<j$. For more on this major theme in model theory see e.g. \cite{tent}. 

\begin{conjecture} For $k\geq 3$, any infinite set-homogeneous $k$-hypergraph with stable theory is complete or has complete complement.
\end{conjecture}

The corresponding result for stable {\em homogeneous} 3-hypergraphs holds, by \cite{lachlan}.
This conjecture would hold if there is no example as in Problem~\ref{homogenizable} for $k$-hypergraphs with $k\geq 3$. Indeed, suppose this is the case, and that $(M,E)$ is a set-homogeneous stable $k$-hypergraph with $k\geq 3$. Then $(M,E)$  can be expanded to a stable homogeneous structure $M'$ over a finite relational language $L'$, with the same automorphism group. By Lachlan's theory for such structures, $M$ will be smoothly approximated by a sequence of finite such structures $(M_i:i\in \mathbb{N})$ which by Theorem~\ref{2trans} will have 2-transitive automorphism group. By the classification of finite 2-transitive permutation groups, it follows  that an infinite group will be interpretable in $M'$, contrary to the main theorem of \cite{mac}.

We also repeat a problem from the introduction of \cite{droste+}, where it is asked whether the two examples $R(3)$ and its complement are the only countably infinite set-homogeneous but not homogeneous graphs.
\begin{problem} Classify countably-infinite set-homogeneous graphs.
\end{problem}

\end{document}